\documentclass{amsart}
  \usepackage{verbatim}
\usepackage{mathrsfs}
\usepackage{latexsym}
\usepackage{epsfig}
\usepackage{amsmath}
\usepackage{bbm}
\usepackage{graphics}
\usepackage{amssymb}
\usepackage{amsthm}
\usepackage{amsrefs}
\usepackage{amsopn}
\usepackage{amscd}
\usepackage[all,knot]{xy}
\xyoption{all}
\usepackage{rotating}
\usepackage{hyperref}

\theoremstyle{plain}
\newtheorem{thm}{Theorem}[section]
\newtheorem{lem}[thm]{Lemma}
\newtheorem{prop}[thm]{Proposition}
\newtheorem{cor}[thm]{Corollary}

\newtheorem{hyp}[thm]{Hypotheses}
\newtheorem*{thm*}{Theorem}
\newtheorem*{lem*}{Lemma}
\newtheorem*{prop*}{Proposition}
\newtheorem*{cor*}{Corollary}

\theoremstyle{definition}
\newtheorem{defn}[thm]{Definition}
\newtheorem*{defn*}{Definition}

\newtheorem{ex}[thm]{Example}
{}
\newtheorem{rem}[thm]{Remark}
\newtheorem*{rem*}{Remark}
\newtheorem{hyp_plain}[thm]{Hypotheses}
\newtheorem{notation}[thm]{Notation}{}
\newtheorem{convention}[thm]{Convention}{}
{}
\newtheorem*{ack}{Acknowledgements}{}

\theoremstyle{remark}
{}
{}
{}

\def\cie{\subseteq}
\def\notcie{\nsubseteq}

\def\iso{\cong}

\def\Un{\bigcup}
\def\un{\cup}

\def\intersec{\cap}

\def\to{\longrightarrow}

\def\rimp{\Rightarrow}

\def\nat{\mathbb{N}}

\def\int{\mathbb{Z}}

\def\str{\mathcal{O}}

\def\a{\alpha}

\def\e{\epsilon}
\def\l{\lambda}
\def\r{\rho}
\def\s{\sigma}
\def\S{\Sigma}
\def\t{\tau}

\def\module{\mathrm{mod}}

\def\Module{\mathrm{Mod}}

\DeclareMathOperator{\End}{End}
\DeclareMathOperator{\Spec}{Spec}
\DeclareMathOperator{\Sing}{Sing}

\DeclareMathOperator{\Spc}{Spc}

\DeclareMathOperator{\supp}{supp}

\DeclareMathOperator{\id}{id}
\DeclareMathOperator{\im}{im}
\DeclareMathOperator{\Hom}{Hom}

\DeclareMathOperator{\hocolim}{hocolim}

\DeclareMathOperator{\QCoh}{QCoh}

\DeclareMathOperator{\Vis}{Vis}

\title{Support theory via actions of tensor triangulated categories}
\author{Greg Stevenson}
\address{Universit\"at Bielefeld, Fakult\"at f\"ur Mathematik, BIREP Gruppe, Postfach 10\,01\,31, 33501 Bielefeld, Germany.}
\email{gstevens@math.uni-bielefeld.de}

\begin{document}



\begin{abstract}
\noindent We give a definition of the action of a tensor triangulated category $\mathcal{T}$ on a triangulated category $\mathcal{K}$. In the case that $\mathcal{T}$ is rigidly-compactly generated and $\mathcal{K}$ is compactly generated we show this gives rise to a notion of supports which categorifies work of Benson, Iyengar, and Krause and extends work of Balmer and Favi. We prove that a suitable version of the local-to-global principle holds very generally. A relative version of the telescope conjecture is formulated and we give a sufficient condition for it to hold. 
\end{abstract}

\maketitle

\tableofcontents

\section{Introduction}
Triangulated categories, introduced by Verdier \cite{VerdierThesis} and by Dold and Puppe \cite{DoldPuppe} (but without Verdier's octahedral axiom), permeate modern mathematics. Their utility has been demonstrated in algebraic geometry, motivic theory, homotopy theory, modular representation theory, and noncommutative geometry: the theory of Grothendieck duality (\cite{HartshorneRD}, \cite{IKTAC}, \cite{MurfetThesis}, \cite{NeeGrot}, \cite{NeeFlat}), Voevodsky's motivic category (\cite{LMC}, \cite{AyoubT}), Devinatz, Hopkins, and Smith's work on tensor nilpotence \cite{DevinatzHopkinsSmith}, support varieties and the extension of complexity to infinitely generated representations (\cite{CDW}, \cite{BCRInf1}, \cite{BCRInf2}), and recent work on the Baum-Connes conjecture \cite{DellAmbrogioKK} respectively are striking examples of the applications of triangulated categories in these areas.

In each of these areas one often has the good fortune to have more than just a triangulated category. Indeed, usually the triangulated categories arising are naturally \emph{tensor triangulated categories}: we say $(\mathcal{T},\otimes,\mathbf{1})$ is tensor triangulated if $\mathcal{T}$ is a triangulated category and $(\otimes,\mathbf{1})$ is a symmetric monoidal structure on $\mathcal{T}$ such that $\otimes$ is exact in each variable and preserves any coproducts $\mathcal{T}$ might possess. This is a very rich structure and exploiting the monoidal product leads to many beautiful results such as the work of Neeman \cite{NeeChro} and Thomason \cite{Thomclass} on the classification of thick subcategories of derived categories of perfect complexes in algebraic geometry.

Tensor triangular geometry, developed by Paul Balmer \cite{BaSpec}, \cite{BaSSS}, \cite{BaFilt}, \cite{BaRickard}, 
 associates to any essentially small tensor triangulated category $(\mathcal{T},\otimes,\mathbf{1})$ a topological space $\Spc \mathcal{T}$, the \emph{spectrum} of $\mathcal{T}$.  The spectrum comes with a universal, tensor compatible, support theory which assigns to objects of $\mathcal{T}$ closed subsets of the spectrum. This generalizes the homological support for derived categories of sheaves in algebraic geometry and the support varieties attached to representations in modular representation theory. One obtains from this support theory a classification of $\otimes$-ideals which unifies classifications occurring in algebraic geometry, modular representation theory, and algebraic topology.

Now suppose $(\mathcal{T},\otimes,\mathbf{1})$ is a compactly generated tensor triangulated category and the compact objects form a tensor subcategory. In \cite{BaRickard} Balmer and Favi have used tensor idempotents built from support data on the spectrum $\Spc \mathcal{T}^c$ of the compact objects $\mathcal{T}^c$ to extend Balmer's notion of supports to $\mathcal{T}$. A related construction due to Benson, Iyengar, and Krause \cite{BIK} takes as input an $R$-linear compactly generated  triangulated category $\mathcal{K}$, where $R$ is a (graded) commutative noetherian ring, and assigns supports valued in $\Spec R$ to objects of $\mathcal{K}$. Our aim is to develop relative tensor triangular geometry by allowing a tensor triangulated category $\mathcal{T}$ to act on $\mathcal{K}$ i.e., there is a biexact functor $\mathcal{T}\times\mathcal{K} \to \mathcal{K}$ which is  compatible with the monoidal structure on $\mathcal{T}$ and associative and unital in the appropriate senses. This can be viewed as a categorification of the work of Benson, Iyengar, and Krause; for instance, letting $R$ be a commutative noetherian ring, an action of the unbounded derived category $D(R)$ yields the same support theory as the support construction of \cite{BIK}. Furthermore, one can view it as extending this construction to noetherian separated schemes. By construction it specializes to the theory of Balmer and Favi when a tensor triangulated category acts on itself in the obvious way. Thus the notion of action provides a link between these two theories of supports and we are able to extend many of the important results of both theories to the case of actions.

Let us fix compactly generated triangulated categories $\mathcal{T}$ and $\mathcal{K}$. Furthermore, suppose $\mathcal{T}$ carries a compatible symmetric monoidal structure $(\mathcal{T},\otimes,\mathbf{1})$ so that the compact objects form a rigid tensor triangulated subcategory $(\mathcal{T}^c,\otimes,\mathbf{1})$ whose spectrum $\Spc \mathcal{T}^c$ is a noetherian topological space (these hypotheses are not necessary for all of the results we quote but are chosen for simplicity). We recall that $\mathcal{T}^c$ is rigid if for all $x$ and $y$ in $\mathcal{T}^c$, setting $x^{\vee} = \hom(x,\mathbf{1})$, the natural map
\begin{displaymath}
x^{\vee}\otimes y \to \hom(x,y)
\end{displaymath}
is an isomorphism, where $\hom(-,-)$ denotes the internal hom which is guaranteed to exist in this case by Brown representability.  In Section \ref{sec_actions} we give a definition of a left action $(-)*(-)$ of $\mathcal{T}$ on $\mathcal{K}$. To each specialization closed subset $\mathcal{V}\cie \Spc \mathcal{T}^c$ and each point $x\in \Spc \mathcal{T}^c$ we associate $\otimes$-idempotent objects $\mathit{\mathit{\Gamma}}_\mathcal{V}\mathbf{1}$ and $\mathit{\Gamma}_x\mathbf{1}$ of $\mathcal{T}$ as in \cite{BaRickard}. The object $\mathit{\Gamma}_\mathcal{V}\mathbf{1}$ is the idempotent corresponding to acyclization with respect to the smashing subcategory generated by the compact objects supported in $\mathcal{V}$ and we denote by $L_\mathcal{V}\mathbf{1}$ the idempotent corresponding to localization at this category. Then $\mathit{\Gamma}_x\mathbf{1}$ is defined to be $\mathit{\Gamma}_{\mathcal{V}(x)}\mathbf{1}\otimes L_{\mathcal{Z}(x)}\mathbf{1}$ where 
\begin{displaymath}
\mathcal{V}(x) = \overline{\{x\}} \quad and \quad \mathcal{Z}(x) = \{y\in \Spc \mathcal{T}^c \; \vert \; x\notin \mathcal{V}(y)\}.
\end{displaymath}
 We prove in Lemmas \ref{lem_action_transfer1} and \ref{lem_action_transfer} that each specialization closed subset $\mathcal{V}$ yields a localization sequence
\begin{displaymath}
\xymatrix{
\mathit{\Gamma}_\mathcal{V}\mathcal{K} \ar[r]<0.5ex> \ar@{<-}[r]<-0.5ex> & \mathcal{K} \ar[r]<0.5ex> \ar@{<-}[r]<-0.5ex> & L_\mathcal{V}\mathcal{K} 
}
\end{displaymath}
where $\mathit{\Gamma}_\mathcal{V}\mathcal{K}$
is the essential image of $\mathit{\Gamma}_\mathcal{V}\mathbf{1}*(-)$. Furthermore, $\mathit{\Gamma}_\mathcal{V}\mathcal{K}$ is generated by objects of $\mathcal{K}^c$ by Corollary \ref{cor_transfer_cg}.  The idempotents $\mathit{\Gamma}_x\mathbf{1}$ give rise to supports on $\mathcal{K}$ with values in $\Spc\mathcal{T}^c$: for an object $A$ of $\mathcal{K}$ we set
\begin{displaymath}
\supp A = \{x\in \Spc\mathcal{T}^c \; \vert \; \mathit{\Gamma}_x\mathbf{1}*A \neq 0\}.
\end{displaymath}
When $\mathcal{T}$ is rigidly-compactly generated and $\mathcal{K}$ is compactly generated the subcategories $\mathit{\Gamma}_\mathcal{V}\mathcal{K}$ and $L_\mathcal{V}\mathcal{K}$ consist precisely of those objects whose support is in $\mathcal{V}$ and $\Spc\mathcal{T}^c \setminus \mathcal{V}$ respectively and the associated localization triangles decompose objects into a piece supported in each of these subsets; this last fact is proved in Proposition \ref{prop_abs_supp_prop} together with other desirable properties of the support.

The local-to-global principle, originally introduced in \cite{BIKstrat2} in the context of ring actions on triangulated categories, allows one to reduce classification problems to considering local pieces of a triangulated category. We introduce the following version for actions of triangulated categories:
\begin{defn*}[\ref{defn_ltg}]
We say $\mathcal{T}\times\mathcal{K} \stackrel{*}{\to} \mathcal{K}$ satisfies the \emph{local-to-global principle} if for each $A$ in $\mathcal{K}$
\begin{displaymath}
\langle A \rangle_* = \langle \mathit{\Gamma}_x A \; \vert \; x\in \Spc \mathcal{T}^c\rangle_*
\end{displaymath}
where $\langle A \rangle_*$ and $\langle \mathit{\Gamma}_x A \; \vert \; x\in \Spc \mathcal{T}^c\rangle_*$ are the smallest localizing subcategories of $\mathcal{K}$ containing $A$ or the $\mathit{\Gamma}_x A$ respectively and closed under the action of $\mathcal{T}$.
\end{defn*}
Our main result concerning the local-to-global principle is that, assuming $\mathcal{T}$ is sufficiently nice, it is only a property of $\mathcal{T}$ not of the action and it always holds.
\begin{thm*}[\ref{thm_general_ltg}]
Suppose $\mathcal{T}$ is a rigidly-compactly generated tensor triangulated category arising from a monoidal model category and that $\Spc \mathcal{T}^c$ is noetherian. Then the following statements hold:
\begin{itemize}
\item[$(i)$] The local-to-global principle holds for the action of $\mathcal{T}$ on itself;
\item[$(ii)$] The associated support function detects vanishing of objects i.e., $X \in \mathcal{T}$ is zero if and only if $\supp X = \varnothing$;
\item[$(iii)$] For any chain $\{\mathcal{V}_i\}_{i\in I}$ of specialization closed subsets of $\Spc \mathcal{T}^c$ with union $\mathcal{V}$ there is an isomorphism
\begin{displaymath}
\mathit{\Gamma}_\mathcal{V}\mathbf{1} \iso \hocolim \mathit{\Gamma}_{\mathcal{V}_i}\mathbf{1}
\end{displaymath}
where the structure maps are the canonical ones.
\end{itemize}
Furthermore, the relative versions of (i) and (ii) hold for any action of $\mathcal{T}$ on a compactly generated triangulated category $\mathcal{K}$.
\end{thm*}

We also explore a relative version of the telescope conjecture. The telescope conjecture states that if $\mathcal{L}$ is a localizing subcategory of a compactly generated triangulated category $\mathcal{T}$ such that the inclusion of $\mathcal{L}$ admits a coproduct preserving right adjoint i.e., $\mathcal{L}$ is smashing, then $\mathcal{L}$ is generated by compact objects of $\mathcal{T}$. This is a general version of the conjecture originally made for the stable homotopy category of spectra by Bousfield \cite{BousfieldLSH} and Ravenel \cite{RavenelLoc}. It is still open for the stable homotopy category, it is known to be true for certain categories such as the derived category of a noetherian ring (by \cite{NeeChro}), and in the generality we have stated it the conjecture is actually false. For instance Keller has given a counterexample in \cite{KellerSmashing}, although Krause in \cite{KrTele} shows that a slightly weaker version of the conjecture does hold. Our version in the relative setting is as follows:
\begin{defn*}[\ref{defn_relative_tele}]
We say the \emph{relative telescope conjecture} holds for $\mathcal{K}$ with respect to the action of $\mathcal{T}$ if every smashing $\mathcal{T}$-submodule $\mathcal{S}\cie \mathcal{K}$ (this means $\mathcal{S}$ is smashing in $\mathcal{K}$ and $\mathcal{T}\times\mathcal{S} \stackrel{*}{\to} \mathcal{K}$ factors via $\mathcal{S}$  ) is generated as a localizing subcategory by compact objects of $\mathcal{K}$. 
\end{defn*}
We give sufficient conditions for the relative telescope conjecture to hold for the action of $\mathcal{T}$ on $\mathcal{K}$. In order to state one of our results let us introduce the following assignments relating subsets of $\Spc \mathcal{T}^c$ and localizing submodules of $\mathcal{K}$ i.e., those localizing subcategories of $\mathcal{K}$ stable under the action of $\mathcal{T}$.
\begin{defn*}[\ref{defn_vis_sigmatau}]
There are order preserving assignments
\begin{displaymath}
\left\{ \begin{array}{c}
\text{subsets of}\; \Spc\mathcal{T}^c 
\end{array} \right\}
\xymatrix{ \ar[r]<1ex>^\t \ar@{<-}[r]<-1ex>_\s &} \left\{
\begin{array}{c}
\text{localizing submodules} \; \text{of} \; \mathcal{K} \\
\end{array} \right\} 
\end{displaymath}
where for a localizing submodule $\mathcal{L}$ we set
\begin{displaymath}
\s(\mathcal{L}) = \supp \mathcal{L} = \{x \in \Spc\mathcal{T}^c \; \vert \; \mathit{\Gamma}_x\mathbf{1}*\mathcal{L} \neq 0\}
\end{displaymath}
for a subset $W$ of $\Spc \mathcal{T}^c$
\begin{displaymath}
\t(W) = \{A \in \mathcal{K} \; \vert \; \supp A \cie W\}
\end{displaymath}
and both the subsets and subcategories are ordered by inclusion.
\end{defn*}
Our theorem is:
\begin{thm*}[\ref{thm_rel_tele}]
Suppose $\mathcal{T}$ is rigidly-compactly generated, has a monoidal model, and $\Spc\mathcal{T}^c$ is noetherian. Let $\mathcal{T}$ act on a compactly generated triangulated category $\mathcal{K}$ so that the support of any compact object of $\mathcal{K}$ is a specialization closed subset of $\s\mathcal{K}$ and for each irreducible closed subset $\mathcal{V}$ in $\s\mathcal{K}$ there exists a compact object whose support is precisely $\mathcal{V}$. Furthermore, suppose the assignments $\s$ and $\t$ give a bijection between localizing submodules of $\mathcal{K}$ and subsets of $\s\mathcal{K}$. Then the relative telescope conjecture holds for $\mathcal{K}$ i.e., every smashing $\mathcal{T}$-submodule of $\mathcal{K}$ is generated, as a localizing subcategory, by objects compact in $\mathcal{K}$. In particular, if every localizing subcategory of $\mathcal{K}$ is stable under the action of $\mathcal{T}$, for instance if $\mathcal{T}$ is generated as a localizing subcategory by the tensor unit, then the usual telescope conjecture holds for $\mathcal{K}$.
\end{thm*}

We also prove several results that show one can work locally with actions to facilitate computations. Rather than stating the technical results here let us mention that we get a new proof of the following result (see \cite{AJS3} Corollary 4.13 and \cite{BaRickard} Corollary 6.8) which follows painlessly by applying our formalism to the classification results of Neeman and Thomason.

\begin{cor*}[\ref{cor_epicwin}]
Let $X$ be a noetherian scheme. Then, letting $D(X)$ act on itself, the assignments $\s$ and $\t$ give a bijection between subsets of $X$ and localizing $\otimes$-ideals of $D(X)$. Furthermore, the relative telescope conjecture holds.
\end{cor*}

Having stated the main results let us now give a brief outline of the paper. After (very) briefly recalling some preliminary material on tensor triangular geometry in Section \ref{sec_prelims} we give in Section \ref{sec_actions} the definition of a left action and prove some basic technical results concerning generators and the formation of action closed subcategories. In Section \ref{sec_rcg} we restrict to studying actions by rigidly-compactly generated triangulated categories and produce the localization sequences which allow us to define supports in Section \ref{sec_supports}, where we also establish the fundamental properties of the support. Our version of the local-to-global principle is introduced in Section \ref{ssec_ltg} and we prove in Theorem \ref{thm_general_ltg} that, as stated above, it holds for any rigidly-compactly generated triangulated category coming from a monoidal model category. In Section \ref{sec_tele} we define the relative telescope conjecture and prove two general results giving sufficient conditions for action closed smashing subcategories to be generated by compact objects of the ambient category. The penultimate section provides tools for working with actions locally with respect to a cover of the spectrum by quasi-compact opens. In particular we prove that supports can be computed locally and classification of action closed subcategories can be checked locally. Finally, in Section \ref{sec_BIK}, we make precise the sense in which our results relate to the formalism of Benson, Iyengar, and Krause at least in the absence of a grading.

We have now given some details about what is in the paper. However, it is, in this case, important to say a little about what is \emph{not} in the paper. We only provide a single example (Corollary \ref{cor_epicwin}) illustrating the utility of the machinery developed and thus one might be led to wonder if all this formalism is somewhat sterile. We wish to assure the reader that in fact other applications already exists. Parts of the formalism are used in \cite{DSgraded} to give a classification of the localizing $\otimes$-ideals of the derived category of a graded noetherian commutative ring (where the grading can come from any finitely generated abelian group). The motivating application was to give a classification of the localizing subcategories of certain singularity categories. The details of this application can be found in \cite{Stevensonclass}. As an enticement we offer the following:

\begin{thm*}
Let $Q$ be a regular local ring and let $\{q_1,\ldots,q_c\}$ be a regular sequence in $Q$. Set $(R,\mathfrak{m},k) = Q/(q_1,\ldots,q_n)$ and let us assume that
\begin{displaymath}
\dim_k \mathfrak{m}/\mathfrak{m}^2 - \dim R = c.
\end{displaymath}
Denote by $Y$ the hypersurface in $\mathbb{P}^{c-1}_Q$ defined by $\sum_{i=1}^n q_ix_i$ where the $x_i$ are coordinates on $\mathbb{P}^{c-1}_Q$. Then the assignments of Definition \ref{defn_vis_sigmatau} give order preserving bijections
\begin{displaymath}
\left\{ \begin{array}{c}
\text{specialization closed} \\ \text{subsets of}\; \Sing Y
\end{array} \right\}
\xymatrix{ \ar[r]<1ex>^{\tau} \ar@{<-}[r]<-1ex>_{\sigma} &} \left\{
\begin{array}{c}
\text{thick subcategories} \\ \text{of} \; D_{Sg}(R)
\end{array} \right\},
\end{displaymath}
where $D_{Sg}(R) = D^b(R\text{-}\module)/D^{\mathrm{perf}}(R)$ is the singularity category of $R$ and $\Sing Y$ denotes the set of singular points of $Y$.
\end{thm*}

\begin{ack}
This article consists of results taken from my PhD thesis written under the supervision of Amnon Neeman at the Australian National University; naturally I would like to thank Amnon for his interest and support as well as for many stimulating conversations and suggesting many improvements. Many thanks are also due to Paul Balmer and Ivo Dell'Ambrogio for numerous helpful comments on a preliminary version of this manuscript. Finally, I would like to thank the anonymous referee for their careful reading of the paper and valuable comments.
\end{ack}

\section[Tensor triangular geometry]{Preliminaries on Tensor Triangular Geometry}\label{sec_prelims}
We give a very brief introduction, mostly to fix notation, to the aspects of Balmer's tensor triangular geometry, as developed in \cite{BaSpec}, \cite{BaSSS} and \cite{BaRickard}, which will be necessary for our purposes. For the reader who desires a more thorough introduction we recommend the survey article \cite{BaTTG}.

Let $\mathcal{T}$ be an essentially small \emph{tensor triangulated category} i.e., we have a triple $(\mathcal{T},\otimes,\mathbf{1})$ where $(\otimes, \mathbf{1})$ is a symmetric monoidal structure on $\mathcal{T}$ which is exact in each variable. 

We say that a thick i.e., summand closed and triangulated, subcategory $\mathcal{I}$ of $\mathcal{T}$ is a $\otimes$\emph{-ideal} if for all $X\in \mathcal{T}$ and $Y\in \mathcal{I}$ the object $X\otimes Y$ is contained in $\mathcal{I}$. We say that a $\otimes$-ideal $\mathcal{P}$ is \emph{prime} if $X\otimes Y$ lies in $\mathcal{P}$ if and only if one of $X$ or $Y$ is an object of $\mathcal{P}$.

The set of prime ideals of $\mathcal{T}$ is denoted $\Spc \mathcal{T}$ and we refer to it as the \emph{spectrum of} $\mathcal{T}$.

For each $X\in \mathcal{T}$ we define the \emph{support} of $X$ to be the set
\begin{displaymath}
\supp X = \{\mathcal{P}\in \Spc \mathcal{T} \; \vert \; X\notin \mathcal{P}\}.
\end{displaymath}
These subsets constitute a basis of closed subsets for a topology on $\mathcal{T}$ which we call the Zariski topology and from now on we consider the spectrum as a topological space. In \cite{BaSpec} Balmer proves that the spectrum of $\mathcal{T}$ together with this notion of support is universal amongst reasonable support data for objects of $\mathcal{T}$. Furthermore, the support gives rise to a classification of radical $\otimes$-ideals.  In order to state this result we need to recall the notion of Thomason subsets.

\begin{defn}
Let $X$ be a topological space. A subset $V\cie X$ is a \emph{Thomason subset} if it is of the form $V = \Un_i V_i$ where each $V_i$ is a closed subset of $X$ with quasi-compact complement.
\end{defn}

Let us also recall that a $\otimes$-ideal $\mathcal{I}$ is \emph{radical} if $X^{\otimes n}\in \mathcal{I}$ implies that $X\in \mathcal{I}$. We will often omit the word radical i.e., we take it as understood that thick $\otimes$-ideal means radical thick $\otimes$-ideal. For the class of essentially small triangulated categories we will mostly be concerned with, namely rigid tensor triangulated categories, all thick $\otimes$-ideals are radical in any case (see for example \cite{BaSpec} Remark 4.3).

\begin{thm}[\cite{BaSpec} 4.10]
Let $\mathfrak{S}$ denote the set of Thomason subsets of $\Spc \mathcal{T}$ and let $\mathfrak{R}$ denote the set of thick radical $\otimes$-ideals of $\mathcal{T}$. Then there is an order preserving bijection $\mathfrak{S} \stackrel{\sim}{\to} \mathfrak{R}$ given by the assignments
\begin{displaymath}
\mathcal{V} \mapsto \{a\in \mathcal{T} \; \vert \; \supp a \cie \mathcal{V}\} \;\; \text{for} \;\; \mathcal{V}\in \mathfrak{S}
\end{displaymath}
and
\begin{displaymath}
\mathcal{I} \mapsto \Un_{a\in \mathcal{I}} \supp a \;\; \text{for} \;\; \mathcal{I} \in \mathfrak{R}.
\end{displaymath}
\end{thm}

We now wish to consider a rigidly-compactly generated tensor triangulated category $\mathcal{T}$, i.e., $(\mathcal{T},\otimes,\mathbf{1})$ is a compactly generated triangulated category $\mathcal{T}$ together with a symmetric monoidal structure such that the monoidal product $\otimes$ is an exact coproduct preserving functor in each variable and the compact objects $\mathcal{T}^c$ are a rigid tensor subcategory. Rigidity is the condition that each compact object be strongly dualizable, further details are given in Section \ref{sec_rcg}.

Given a Thomason subset $\mathcal{V}\cie \Spc \mathcal{T}^c$ we denote by $\mathcal{T}^c_\mathcal{V}$ the thick subcategory of compact objects supported on $\mathcal{V}$. We let $\mathcal{T}_\mathcal{V}$ be the localizing subcategory generated by $\mathcal{T}^c_\mathcal{V}$ and note that $\mathcal{T}_\mathcal{V}$ is smashing as it is generated by compact objects of $\mathcal{T}$. Let us spend a little time spelling out the consequences of this fact. The subcategory $\mathcal{T}_\mathcal{V}$ gives rise to a smashing localization sequence
\begin{displaymath}
\xymatrix{
\mathit{\Gamma}_\mathcal{V}\mathcal{T} \ar[r]<0.5ex>^(0.6){i_*} \ar@{<-}[r]<-0.5ex>_(0.6){i^!} & \mathcal{T} \ar[r]<0.5ex>^(0.5){j^*} \ar@{<-}[r]<-0.5ex>_(0.5){j_*} & L_\mathcal{V}\mathcal{T}
}
\end{displaymath}
i.e., all four functors are exact and coproduct preserving, $i_*$ and $j_*$ are fully faithful, $i^!$ is right adjoint to $i_*$, and $j_*$ is right adjoint to $j^*$. In particular there are associated coproduct preserving acyclization and localization functors given by $i_*i^!$ and $j_*j^*$ respectively. As in \cite{HPS} Definition 3.3.2 this gives rise to Rickard idempotents which we denote by $\mathit{\Gamma}_\mathcal{V}\mathbf{1}$ and $L_\mathcal{V}\mathbf{1}$ with the property that 
\begin{displaymath}
i_*i^! \iso \mathit{\Gamma}_\mathcal{V}\mathbf{1}\otimes(-) \quad \text{and} \quad j_*j^* \iso L_\mathcal{V}\mathbf{1}\otimes(-).
\end{displaymath}
It follows that they are $\otimes$-orthogonal by the usual properties of localization and acyclization functors. We will also sometimes write $\mathit{\Gamma}_\mathcal{V}\mathcal{T}$ for the category associated to $\mathcal{V}$.

One can go on to define supports for objects of $\mathcal{T}$ taking values in some subset of $\Spc \mathcal{T}^c$. In fact we will wish to consider supports and the associated tensor idempotents but we wait until they are required in Section \ref{sec_supports} to introduce them.

\section{Tensor actions}\label{sec_actions}
To begin we propose a definition of what it means for a tensor triangulated category to act on another triangulated category. We define here the notion of left action and express a sinistral bias by only considering left actions and referring to them simply as actions.

\begin{convention}
Throughout by a \emph{tensor triangulated category} $(\mathcal{T},\otimes,\mathbf{1})$ we mean a triangulated category $\mathcal{T}$ together with a symmetric monoidal structure such that the monoidal product $\otimes$ is an exact functor in each variable. We also require that $\otimes$ preserves whatever coproducts $\mathcal{T}$ might have and interacts well with the suspension as in Definition \ref{defn_action} (3). We do not assume, unless explicitly stated, that the triangulated categories we deal with are essentially small.

By a \emph{compactly generated tensor triangulated category} we mean a tensor triangulated category as above which is compactly generated and such that the compact objects form a tensor subcategory.
\end{convention}

\begin{defn}\label{defn_action}
Let $(\mathcal{T},\otimes,\mathbf{1})$ be a tensor triangulated category and $\mathcal{K}$ a triangulated category. A \emph{left action} of $\mathcal{T}$ on $\mathcal{K}$ is a functor 
\begin{displaymath}
*\colon \mathcal{T}\times \mathcal{K} \to \mathcal{K}
\end{displaymath}
which is exact in each variable, i.e.\ for all $X\in \mathcal{T}$ and $A\in \mathcal{K}$ the functors $X*(-)$ and $(-)*A$ are exact (such a functor is called \emph{biexact}), together with natural isomorphisms
\begin{displaymath}
a_{X,Y,A}\colon (X\otimes Y)*A \stackrel{\sim}{\to} X*(Y*A)
\end{displaymath}
and
\begin{displaymath}
l_A\colon \mathbf{1}*A \stackrel{\sim}{\to} A
\end{displaymath}
for all $X,Y\in \mathcal{T}$, $A\in \mathcal{K}$, compatible with the biexactness of $(-)*(-)$ and satisfying the following conditions:
\begin{itemize}
\item[(1)] The associator $a$ satisfies the pentagon condition which asserts that the following diagram commutes for all $X,Y,Z$ in $\mathcal{T}$ and $A$ in $\mathcal{K}$
\begin{displaymath}
\xymatrix{
& X*(Y*(Z* A)) & \\
X*((Y \otimes Z)*A) \ar[ur]^{X*a_{Y,Z,A}} & & (X\otimes Y)*(Z*A) \ar[ul]_{a_{X,Y,Z*A}} \\
(X\otimes(Y\otimes Z))*A \ar[u]^{a_{X,Y\otimes Z,A}} && ((X\otimes Y)\otimes Z)*A \ar[u]_{a_{X\otimes Y, Z,A}} \ar[ll]
}
\end{displaymath}
where the bottom arrow is the associator of $(\mathcal{T},\otimes,\mathbf{1})$.
\item[(2)] The unitor $l$ makes the following squares commute for every $X$ in $\mathcal{T}$ and $A$ in $\mathcal{K}$
\begin{displaymath}
\xymatrix{
X* (\mathbf{1}*A) \ar[r]^(0.6){X * l_A} & X*A \ar[d]^{1_{X*A}} \\
(X\otimes \mathbf{1})*A \ar[u]^{a_{X,\mathbf{1},A}} \ar[r] & X*A
}
\qquad
\xymatrix{
\mathbf{1}*(X*A) \ar[r]^(0.6){l_{X*A}} & X*A \ar[d]^{1_{X*A}} \\
(\mathbf{1}\otimes X)*A \ar[u]^{a_{\mathbf{1},X,A}} \ar[r] & X*A
}
\end{displaymath}
where the bottom arrows are the right and left unitors of $(\mathcal{T},\otimes,\mathbf{1})$.
\item[(3)] For every $A$ in $\mathcal{K}$ and $r,s \in \int$ the diagram
\begin{displaymath}
\xymatrix{
\S^r \mathbf{1}*\S^s A \ar[r]^(0.6)\sim \ar[d]_{\wr} & \S^{r+s} A \ar[d]^{(-1)^{rs}} \\
\S^r (\mathbf{1}* \S^s A) \ar[r]_(0.6)\sim & \S^{r+s}A
}
\end{displaymath}
is commutative, where the left vertical map comes from exactness in the first variable of the action, the bottom horizontal map is the unitor, and the top map is given by the composite
\begin{displaymath}
\S^r\mathbf{1}*\S^s A \to \S^s(\S^r\mathbf{1} * A) \to \S^{r+s}(\mathbf{1}*A) \stackrel{l}{\to} \S^{r+s}A
\end{displaymath}
whose first two maps use exactness in both variables of the action.
\item[(4)] The functor $*$ distributes over coproducts whenever they exist i.e., for families of objects $\{X_i\}_{i\in I}$ in $\mathcal{T}$ and $\{A_j\}_{j\in J}$ in $\mathcal{K}$, and $X$ in $\mathcal{T}$, $A$ in $\mathcal{K}$ the canonical maps
\begin{displaymath}
\coprod_i (X_i * A) \stackrel{\sim}{\to} (\coprod_i X_i)* A
\end{displaymath}
and
\begin{displaymath}
\coprod_j (X*A_j) \stackrel{\sim}{\to} X*(\coprod_j A_j)
\end{displaymath}
are isomorphisms whenever the coproducts concerned, on both the left and the right of each isomorphism, exist.
\end{itemize}
\end{defn}

\begin{rem}
Given composable morphisms $f,f'$ in $\mathcal{T}$ and $g,g'$ in $\mathcal{K}$ one has
\begin{displaymath}
(f'*g')(f*g) = (f'f*g'g)
\end{displaymath}
by functoriality of $\mathcal{T} \times \mathcal{K} \stackrel{*}{\to} \mathcal{K}$.

We also note it follows easily from the definition that both $0_\mathcal{T}*(-)$ and $(-)*0_\mathcal{K}$ are isomorphic to the zero functor.
\end{rem}

We view $\mathcal{K}$ as a module over $\mathcal{T}$ and from now on we will use the terms module and action interchangeably. There are of course, depending on the context, natural notions of $\mathcal{T}$-submodule.

\begin{defn}\label{defn_submodule}
Let $\mathcal{L}\cie \mathcal{K}$ be a localizing (thick) subcategory. We say $\mathcal{L}$ is a localizing (thick) \emph{$\mathcal{T}$-submodule} of $\mathcal{K}$ if the functor
\begin{displaymath}
\mathcal{T}\times \mathcal{L} \stackrel{*}{\to} \mathcal{K}
\end{displaymath}
factors via $\mathcal{L}$ i.e., $\mathcal{L}$ is closed under the action of $\mathcal{T}$. We note that in the case $\mathcal{K} = \mathcal{T}$ acts on itself by $\otimes$ this gives the notion of a localizing (thick) $\otimes$-ideal of $\mathcal{T}$. By a smashing or compactly generated (by compact objects in the ambient category) submodule we mean the obvious things.
\end{defn}


\begin{notation}\label{not_smallestsubcat}
For a collection of objects $\mathcal{A}$ in $\mathcal{K}$ we denote by $\langle \mathcal{A} \rangle$ the smallest localizing subcategory containing $\mathcal{A}$ and by $\langle \mathcal{A} \rangle_*$ the smallest localizing $\mathcal{T}$-submodule of $\mathcal{K}$ containing $\mathcal{A}$. 

Given also a collection of objects $\mathcal{X}$ of $\mathcal{T}$ we denote by
\begin{displaymath}
\mathcal{X}*\mathcal{A} = \langle X*A\; \vert \; X\in \mathcal{X}, A\in \mathcal{A} \rangle_*
\end{displaymath}
the localizing submodule generated by products of the objects from $\mathcal{X}$ and $\mathcal{A}$.
\end{notation}

\begin{rem}
We do not introduce similar notation for thick subcategories as we will almost invariably work with localizing subcategories. However, it is worth noting that the formal results about submodules we shall prove are generally also true or have obvious analogues for thick submodules.
\end{rem}

\begin{hyp_plain}
From this point forward we assume that both $\mathcal{T}$ and $\mathcal{K}$ have all set-indexed coproducts.
\end{hyp_plain}

The operation of forming submodules is well behaved. We will show below that it commutes with the action in an appropriate sense. Most important for us is the fact that given generating sets for $\mathcal{L}\cie \mathcal{T}$ and $\mathcal{M}\cie \mathcal{K}$ we obtain a generating set for $\mathcal{L}*\mathcal{M}$ as a submodule. First we prove a general lemma (not involving actions), encompassing various standard arguments in the literature, concerning closure properties with respect to families of exact coproduct preserving functors.

\begin{lem}\label{lem_general_closure}
Suppose $\mathcal{R}$ and $\mathcal{S}$ are triangulated categories having enough coproducts, $\{F_\l\}_{\l\in \Lambda}$ is a family of coproduct preserving exact functors $\mathcal{R} \to \mathcal{S}$, and $\mathcal{M}$ is a localizing subcategory of $\mathcal{S}$. Then the full subcategory 
\begin{displaymath}
\mathcal{L} = \{X\in \mathcal{R} \; \vert \; F_\l(X)\in \mathcal{M} \;\; \forall \l \in \Lambda\}
\end{displaymath}
is a localizing subcategory of $\mathcal{R}$. In particular, if $\mathcal{C}$ is a collection of objects of $\mathcal{R}$ such that for all $\l\in \Lambda$ we have $F_\l(\mathcal{C})\cie \mathcal{M}$ then every object $C\in \langle \mathcal{C}\rangle$ satisfies $F_\l(C)\in \mathcal{M}$ for all $\l\in \Lambda$.
\end{lem}
\begin{proof}
We can write $\mathcal{L}$ as an intersection of localizing subcategories,
\begin{displaymath}
\mathcal{L} = \bigcap_{\l\in \Lambda}\ker(\mathcal{R}\stackrel{F_\l}{\to}\mathcal{S} \to \mathcal{S}/\mathcal{M}),
\end{displaymath}
so it is immediate that $\mathcal{L}$ is localizing. The second statement then follows. By hypothesis $\mathcal{C}\cie \mathcal{L}$ and since $\mathcal{L}$ is localizing we deduce that $\langle\mathcal{C}\rangle \cie \mathcal{L}$.
\end{proof}

The next lemma makes explicit the way in which we will use this rather general result in the case that $\mathcal{T}$ acts on $\mathcal{K}$.

\begin{lem}\label{lem_general_closure2}
Suppose that $\mathcal{A}$ is a collection of objects of $\mathcal{K}$ such that $\mathcal{A}$ is stable under the action of $\mathcal{T}$. Then $\langle \mathcal{A} \rangle$ is a localizing $\mathcal{T}$-submodule. Similarly, if $\mathcal{X}$ is a collection of objects of $\mathcal{T}$ and $\mathcal{N}$ is a localizing subcategory of $\mathcal{K}$ closed under the action of the objects in $\mathcal{X}$ then $\mathcal{M}$ is closed under the action of $\langle \mathcal{X} \rangle$.
\end{lem}
\begin{proof}
In order to prove the first statement one simply applies the previous lemma in the case that $\mathcal{R}=\mathcal{S}=\mathcal{K}$, the family of functors is $\{X*(-)\;\vert\; X\in \mathcal{T}\}$, and $\mathcal{M} = \langle \mathcal{A} \rangle$. The second statement is also immediate by applying the lemma appropriately.
\end{proof}

\begin{lem}\label{lem_alreadysubmod}
Suppose $\mathcal{I} \cie \mathcal{T}$ is a localizing $\otimes$-ideal and $\mathcal{A}$ is a collection of objects of $\mathcal{K}$. Then there is an equality of localizing submodules of $\mathcal{K}$
\begin{displaymath}
\mathcal{I}*\mathcal{A} = \langle X*A \; \vert \; X\in \mathcal{I}, A\in \mathcal{A} \rangle.
\end{displaymath}
\end{lem}
\begin{proof}
As $\mathcal{I}$ is a $\otimes$-ideal the collection of objects
\begin{displaymath}
\{X*A \; \vert \; X\in \mathcal{I}, A\in \mathcal{A}\}
\end{displaymath}
is, by associativity, closed under the action of $\mathcal{T}$. So by the last lemma the subcategory $\langle X*A \; \vert \; X\in \mathcal{I}, A\in \mathcal{A} \rangle$ is a localizing submodule from which the claimed equality is immediate.
\end{proof}

\begin{lem}\label{lem_loc_commutes}
Formation of localizing subcategories commutes with the action, i.e., given a set of objects $\mathcal{X}$ of $\mathcal{T}$ and a set of objects $\mathcal{A}$ of $\mathcal{K}$
\begin{displaymath}
\langle Y*B \; \vert \; Y \in \langle \mathcal{X} \rangle, B \in \langle \mathcal{A} \rangle \rangle = \langle X*A\; \vert \; X\in \mathcal{X}, A\in \mathcal{A} \rangle.
\end{displaymath}
\end{lem}
\begin{proof}
Denote the category on the left by $\mathcal{L}$ and the one on the right by $\mathcal{M}$. It is clear $\mathcal{M} \cie \mathcal{L}$. For the converse it is sufficient to check that $\mathcal{M}$ contains generators for $\mathcal{L}$. For each $A\in \mathcal{A}$ define a subcategory
\begin{displaymath}
\mathcal{T}_A = \{Y \in \mathcal{T} \; \vert \; Y*A \in \mathcal{M} \}.
\end{displaymath}
The subcategory $\mathcal{T}_A$ is localizing by Lemma \ref{lem_general_closure}. As, by definition, $X*A$ is in $\mathcal{M}$ for all $X \in \mathcal{X}$ we see each such $X$ lies in $\mathcal{T}_A$. So for any $Y$ in $\langle \mathcal{X} \rangle$ we have $Y$ in $\mathcal{T}_A$. In particular, $Y*A$ lies in $\mathcal{M}$ for each such $Y$ and all $A\in \mathcal{A}$.

Now consider the subcategory
\begin{displaymath}
\{B \in \mathcal{K} \; \vert \; Y * B \in \mathcal{M} \;\; \text{for all}\;\; Y \in \langle \mathcal{X}\rangle\}.
\end{displaymath}
It is localizing and by what we have just seen it contains $\mathcal{A}$. Thus it contains $\langle \mathcal{A} \rangle$ so for every $Y$ in $\langle \mathcal{X}\rangle$ and every $B$ in $\langle \mathcal{A}\rangle$ we have $Y*B$ in $\mathcal{M}$. Hence $\mathcal{M}$ contains generators for $\mathcal{L}$ which gives the equality $\mathcal{L} = \mathcal{M}$.
\end{proof}

Using this result we can prove an analogue for submodules.

\begin{lem}\label{lem_loc_commutes2}
Formation of localizing $\mathcal{T}$-submodules commutes with the action i.e., given a collection of objects $\mathcal{X}$ of $\mathcal{T}$ and a collection of objects $\mathcal{A}$ of $\mathcal{K}$ we have
\begin{align*}
\langle \mathcal{X}\rangle_{\otimes} * \langle \mathcal{A}\rangle &= \langle \mathcal{X}\rangle* \langle \mathcal{A}\rangle \\
&= \mathcal{X} * \mathcal{A}\\
&= \langle Z*(X*A) \; \vert \; Z\in \mathcal{T}, X\in \mathcal{X}, A\in \mathcal{A} \rangle.
\end{align*}
\end{lem}
\begin{proof}
The containment $\langle \mathcal{X}\rangle_{\otimes} * \langle \mathcal{A}\rangle \supseteq \langle \mathcal{X}\rangle* \langle \mathcal{A}\rangle$ is clear. On the other hand we know 
\begin{displaymath}
\langle \mathcal{X}\rangle_\otimes = \mathcal{T}\otimes \langle \mathcal{X}\rangle = \langle Z\otimes Y\; \vert \; Z\in \mathcal{T}, Y\in \langle \mathcal{X}\rangle\rangle = \langle Z\otimes X\; \vert \; Z\in \mathcal{T}, X\in \mathcal{X}\rangle.
\end{displaymath}
the second equality by Lemma \ref{lem_alreadysubmod} and the third by the last lemma. Using this we see that
\begin{align*}
\langle \mathcal{X} \rangle_\otimes * \langle \mathcal{A} \rangle &= \langle Y*B \; \vert \; Y\in \langle Z\otimes X\; \vert \; Z\in \mathcal{T}, X\in \mathcal{X}\rangle, B\in \langle \mathcal{A} \rangle\rangle \\
&= \langle Z*( X*A) \; \vert \; Z\in \mathcal{T}, X\in \mathcal{X}, A\in \mathcal{A}\rangle \\ &\cie \langle \mathcal{X}\rangle *\langle \mathcal{A}\rangle
\end{align*}
the first equality again by Lemma \ref{lem_alreadysubmod} and the second from the last lemma and associativity of the action. This proves the first and third equalities in the statement of the lemma.

The second follows from Lemma \ref{lem_loc_commutes} as it identifies the smallest localizing subcategories containing generators (as submodules) for the submodules in question and hence the smallest submodules containing these generating sets.
\end{proof}

We record here the following trivial observation which turns out to be quite useful.

\begin{lem}\label{lem_gen_locpreserving}
If $\mathcal{T}$ is generated as a localizing subcategory by the tensor unit $\mathbf{1}$ then every localizing subcategory of $\mathcal{K}$ is a $\mathcal{T}$-submodule.
\end{lem}
\begin{proof}
This is immediate from Lemma \ref{lem_general_closure2}.
\end{proof}

\section[Rigidly-compactly generated categories]{The case of rigidly-compactly generated tensor triangulated categories}\label{sec_rcg}
We now restrict ourselves to the case that $(\mathcal{T},\otimes,\mathbf{1})$ is a rigidly-compactly generated tensor triangulated category (unless explicitly mentioned otherwise) acting on a compactly generated $\mathcal{K}$. Actions of such categories have desirable properties and we can extend much of the machinery developed in  \cite{BaRickard}, \cite{BIK}, and \cite{BIKstrat2} to this setting.  First let us make explicit our hypotheses on $\mathcal{T}$.

\begin{defn}
A \emph{rigidly-compactly generated tensor triangulated category} is a compactly generated tensor triangulated category (as usual the monoidal structure is assumed to be symmetric, biexact, and preserve coproducts so that $\mathcal{T}$ has an internal hom by Brown representability which we denote by $\hom(-,-)$) such that $\mathcal{T}^{\mathrm{c}}$, the (essentially small) subcategory of compact objects, is a rigid tensor triangulated subcategory. We recall that $\mathcal{T}^c$ is a rigid tensor triangulated subcategory if the monoidal structure and internal hom restrict to $\mathcal{T}^c$ (in particular the unit object $\mathbf{1}$ must be compact), and for all $x$ and $y$ in $\mathcal{T}^c$, setting $x^{\vee} = \hom(x,\mathbf{1})$, the natural map
\begin{displaymath}
x^{\vee}\otimes y \to \hom(x,y)
\end{displaymath}
is an isomorphism. In particular such categories are almost unital algebraic stable homotopy categories in the sense of \cite{HPS} Definition 1.1.4 (we do not assume the strong compatibility conditions between the tensor, internal hom, and triangles, but in order to also define a notion of cosupport for actions following \cite{BIKcosupp} such conditions are likely desirable).
\end{defn}

In the case that $\mathcal{T}$ is rigidly-compactly generated we can use $\Spc \mathcal{T}^{c}$, as defined in \cite{BaSpec}, in order to define a theory of supports by using the localizing $\otimes$-ideals of $\mathcal{T}$ generated by objects of $\mathcal{T}^c$ as in \cite{BaRickard}. 
 
Our first task is to show that if such a $\mathcal{T}$ acts on a compactly generated triangulated category $\mathcal{K}$ we can obtain, from Rickard idempotents on $\mathcal{T}$, localization sequences on $\mathcal{K}$ where the category of acyclic objects is compactly generated by compact objects of $\mathcal{K}$. 

\begin{convention}
Throughout this section all submodules will be localizing unless explicitly mentioned otherwise.
\end{convention}

We now prove that from a Thomason subset of $\Spc \mathcal{T}^c$ we can produce a pair of compactly generated subcategories of $\mathcal{K}$. We do this via a series of relatively straightforward lemmas.

\begin{lem}\label{lem_action_transfer1}
Suppose $\mathcal{V} \cie \Spc \mathcal{T}^c$ is a Thomason subset. Then the subcategory
\begin{displaymath}
\mathit{\Gamma}_\mathcal{V}\mathcal{K} := \{A \in \mathcal{K} \;  \vert \; \exists A' \; \text{with} \; A\iso \mathit{\Gamma}_\mathcal{V}\mathbf{1}*A' \}
\end{displaymath}
is a localizing $\mathcal{T}$-submodule.
\end{lem}
\begin{proof}
We begin by showing $\mathit{\Gamma}_\mathcal{V}\mathcal{K}$ is localizing. It is sufficient to show that
\begin{displaymath}
\mathit{\Gamma}_\mathcal{V}\mathcal{K} = \ker L_\mathcal{V}\mathbf{1}*(-),
\end{displaymath}
as the kernel of any exact coproduct preserving functor is a localizing subcategory. By \cite{BaRickard} Theorem 3.5 the subcategory $\mathit{\Gamma}_\mathcal{V}\mathcal{T}$ of $\mathcal{T}$ is precisely the essential image, $\im (\mathit{\Gamma}_\mathcal{V}\mathbf{1}\otimes (-))$, of tensoring with $\mathit{\Gamma}_\mathcal{V}\mathbf{1}$ and the corresponding idempotents are tensor orthogonal i.e., $\mathit{\Gamma}_\mathcal{V}\mathbf{1}\otimes L_\mathcal{V}\mathbf{1} = 0$. So if $A$ is in $\mathit{\Gamma}_\mathcal{V}\mathcal{K}$ then
\begin{align*}
L_\mathcal{V}\mathbf{1}* A &\iso L_\mathcal{V}\mathbf{1}*(\mathit{\Gamma}_\mathcal{V}\mathbf{1}*A') \\
&\iso (L_\mathcal{V}\mathbf{1}\otimes \mathit{\Gamma}_\mathcal{V}\mathbf{1})*A' \\
&\iso 0
\end{align*}
showing
\begin{displaymath}
\mathit{\Gamma}_\mathcal{V}\mathcal{K} \cie \ker L_\mathcal{V}\mathbf{1}*(-).
\end{displaymath}
Conversely, suppose $L_\mathcal{V}\mathbf{1}*A = 0$. Then applying $(-)*A$ to the localization triangle
\begin{displaymath}
\mathit{\Gamma}_\mathcal{V}\mathbf{1} \to \mathbf{1} \to L_\mathcal{V}\mathbf{1} \to \S \mathit{\Gamma}_\mathcal{V}\mathbf{1}
\end{displaymath}
in $\mathcal{T}$ we deduce an isomorphism $\mathit{\Gamma}_\mathcal{V}\mathbf{1}*A \stackrel{\sim}{\to} A$. Thus $A$ is in $\mathit{\Gamma}_\mathcal{V}\mathcal{K}$ so  the two subcategories of $\mathcal{K}$ in question are equal as claimed. As stated above this proves $\mathit{\Gamma}_\mathcal{V}\mathcal{K}$ is localizing.

To see it is a submodule note that for $X$ in $\mathcal{T}$ and $A$ in $\mathit{\Gamma}_\mathcal{V}\mathcal{K}$ we have
\begin{align*}
X*A &\iso X*(\mathit{\Gamma}_\mathcal{V}\mathbf{1}*A') \\
&\iso (X\otimes \mathit{\Gamma}_\mathcal{V}\mathbf{1})*A' \\
&\iso (\mathit{\Gamma}_\mathcal{V}\mathbf{1} \otimes X)*A' \\
&\iso \mathit{\Gamma}_\mathcal{V}\mathbf{1} *(X*A').
\end{align*}
\end{proof}

\begin{lem}\label{lem_action_transfer}
Suppose $\mathcal{V}$ is a Thomason subset of $\Spc \mathcal{T}^c$. The subcategory $\mathit{\Gamma}_\mathcal{V}\mathcal{K}$ and the subcategory
\begin{displaymath}
L_\mathcal{V}\mathcal{K} := \{A \in \mathcal{K} \; \vert \; \exists A' \; \text{with} \; A\iso L_\mathcal{V}\mathbf{1}*A' \}
\end{displaymath}
give rise to a localization sequence
\begin{displaymath}
\xymatrix{
\mathit{\Gamma}_\mathcal{V}\mathcal{K} \ar[r]<0.5ex> \ar@{<-}[r]<-0.5ex> & \mathcal{K} \ar[r]<0.5ex> \ar@{<-}[r]<-0.5ex> & L_\mathcal{V}\mathcal{K}
}
\end{displaymath}
i.e, the top row consists of a fully faithful inclusion and the Verdier quotient by its image and both of these functors have right adjoints, and $L_\mathcal{V}\mathcal{K}$ is also a localizing $\mathcal{T}$-submodule.
\end{lem}
\begin{proof}
The statement that $L_\mathcal{V}\mathcal{K}$ is a submodule follows in exactly the same way as for $\mathit{\Gamma}_\mathcal{V}\mathcal{K}$ in the proof of Lemma \ref{lem_action_transfer1}.

So let us demonstrate we have the claimed localization sequence. By definition there is a triangle in $\mathcal{T}$
\begin{displaymath}
\mathit{\Gamma}_\mathcal{V}\mathbf{1} \to \mathbf{1} \to L_\mathcal{V}\mathbf{1} \to \S\mathit{\Gamma}_\mathcal{V}\mathbf{1}
\end{displaymath}
associated to $\mathcal{V}$. For any $A$ in $\mathcal{K}$ the action thus gives us functorial triangles
\begin{displaymath}
\mathit{\Gamma}_\mathcal{V}\mathbf{1}*A \to A \to L_\mathcal{V}\mathbf{1}*A \to \S\mathit{\Gamma}_\mathcal{V}\mathbf{1}*A.
\end{displaymath}
So to prove we have the desired localization sequence it is sufficient to demonstrate 
\begin{displaymath}
L_\mathcal{V}\mathcal{K} = \mathit{\Gamma}_\mathcal{V}\mathcal{K}^{\perp}
\end{displaymath}
by Lemma 3.1 of \cite{BondalReps}. 

We first show $L_\mathcal{V}\mathcal{K} \supseteq \mathit{\Gamma}_\mathcal{V}\mathcal{K}^{\perp}$. Suppose $A\in \mathit{\Gamma}_\mathcal{V}\mathcal{K}^\perp$ and consider the triangle
\begin{displaymath}
\mathit{\Gamma}_\mathcal{V}\mathbf{1}*A \to A \to L_\mathcal{V}\mathbf{1}*A \to \S\mathit{\Gamma}_\mathcal{V}\mathbf{1}*A.
\end{displaymath}
By hypothesis the morphism $\mathit{\Gamma}_\mathcal{V}\mathbf{1}*A \to A$ must be zero so the triangle splits yielding
\begin{displaymath}
L_\mathcal{V}\mathbf{1}*A \iso A \oplus \S \mathit{\Gamma}_\mathcal{V}\mathbf{1}*A.
\end{displaymath}
As $L_\mathcal{V}\mathcal{K}$ is localizing, and thus thick, it must contain $\mathit{\Gamma}_\mathcal{V}\mathbf{1}*A$ i.e., there is some $A'$ in $\mathcal{K}$ such that $\mathit{\Gamma}_\mathcal{V}\mathbf{1}*A \iso L_\mathcal{V}\mathbf{1}*A'$. Hence there are isomorphisms
\begin{align}\label{blargh}
\mathit{\Gamma}_\mathcal{V}\mathbf{1}*A \iso \mathit{\Gamma}_\mathcal{V}\mathbf{1}*(\mathit{\Gamma}_\mathcal{V}\mathbf{1}*A) &\iso \mathit{\Gamma}_\mathcal{V}\mathbf{1}* (L_\mathcal{V}\mathbf{1}*A') \\
&\iso (\mathit{\Gamma}_\mathcal{V}\mathbf{1} \otimes L_\mathcal{V}\mathbf{1})*A' \iso 0
\end{align}
where we have used tensor orthogonality of the Rickard idempotents. Thus $L_\mathcal{V}\mathbf{1}*A \iso A$ is in $L_\mathcal{V}\mathcal{K}$.

It remains to check the containment $L_\mathcal{V}\mathcal{K} \cie \mathit{\Gamma}_\mathcal{V}\mathcal{K}^\perp$. Let $A$ be an object of $\mathit{\Gamma}_\mathcal{V}\mathcal{K}$ and $B$ an object of $L_\mathcal{V}\mathcal{K}$. Observe that as $A$ is in $\mathit{\Gamma}_\mathcal{V}\mathcal{K}$ and $B$ is in $L_\mathcal{V}\mathcal{K}$ we have $L_\mathcal{V}\mathbf{1}*A\iso 0$ and $\mathit{\Gamma}_\mathcal{V}\mathbf{1}*B \iso 0$. Indeed, by symmetry of the monoidal structure on $\mathcal{T}$ the objects $L_\mathcal{V}\mathbf{1}*A$ and $\mathit{\Gamma}_\mathcal{V}\mathbf{1}*B$ lie in both $\mathit{\Gamma}_\mathcal{V}\mathcal{K}$ and $L_\mathcal{V}\mathcal{K}$. It follows they must vanish by orthogonality of the tensor idempotents $\mathit{\Gamma}_\mathcal{V}\mathbf{1}$ and $L_\mathcal{V}\mathbf{1}$ as in (1) and (2) above. So for $f\in \Hom(A,B)$ we obtain via functoriality a map of triangles
\begin{displaymath}
\xymatrix{
\mathit{\Gamma}_\mathcal{V}\mathbf{1}*A \ar[r]^(0.6){\sim} \ar[d] & A \ar[r] \ar[d]^f & 0 \ar[d]\\
0 \ar[r] & B \ar[r]_(0.4){\sim} & L_\mathcal{V}\mathbf{1}*B
}
\end{displaymath}
which shows $f=0$. Hence
\begin{displaymath}
L_\mathcal{V}\mathcal{K} \cie \mathit{\Gamma}_\mathcal{V}\mathcal{K}^{\perp}
\end{displaymath}
proving the equality of these two subcategories. As stated above this yields the desired localization sequence by Lemma 3.1 of \cite{BondalReps}.

\end{proof}

\begin{notation}
We will be somewhat slack with notation and often write, for $A$ in $\mathcal{K}$, $\mathit{\Gamma}_\mathcal{V}A$ rather than $\mathit{\Gamma}_\mathcal{V}\mathbf{1}*A$ when it is clear from the context what we mean. When working with objects $X$ of $\mathcal{T}$ we will use the idempotent notation for the localization and acyclization functors, e.g.\ $\mathit{\Gamma}_\mathcal{V}\mathbf{1}\otimes X$, so no confusion should be possible.
\end{notation}

The next lemma is the first of several results showing rigidly-compactly generated tensor triangulated categories are not just lovely categories in their own right, but they also act well on other compactly generated categories.

\begin{lem}\label{lem_rigid_goodaction}
Suppose $\mathcal{T}\times \mathcal{K} \stackrel{*}{\to} \mathcal{K}$ is an action where $\mathcal{T}$ is rigidly-compactly generated and $\mathcal{K}$ is compactly generated. Then the action restricts to an action at the level of compact objects $\mathcal{T}^c\times \mathcal{K}^c \stackrel{*}{\to} \mathcal{K}^c$.
\end{lem}
\begin{proof}
Let $t$ be a compact object of $\mathcal{T}$. As $\mathcal{T}^c$ is rigid the object $t$ admits a strong dual i.e., there is an object $t^{\vee}$ together with morphisms
\begin{displaymath}
\eta_t \colon \mathbf{1} \to t^{\vee} \otimes t \quad \text{and} \quad \e_t \colon t\otimes t^\vee \to \mathbf{1}
\end{displaymath}
such that the composite
\begin{displaymath}
\xymatrix{
t \ar[r]^{\r_t^{-1}} & t\otimes \mathbf{1} \ar[r]^(0.4){t\otimes \eta_t} & t\otimes (t^\vee \otimes t) \ar[r]^{\a} & (t\otimes t^\vee)\otimes t \ar[r]^(0.6){\e_t \otimes t} & \mathbf{1}\otimes t \ar[r]^(0.6){\l_t} & t
}
\end{displaymath}
where $\r_t$, $\l_t$, and $\a$ are the right and left unitors and the associator for $\mathcal{T}$, is the identity and similarly for $t^\vee$ . Using these maps together with the unitor $l$ and associator $a$ for the action we define natural transformations
\begin{displaymath}
\eta_t' \colon \xymatrix{\id_\mathcal{K} \ar[r]^{l^{-1}}& \mathbf{1}* \ar[r]^(0.4){\eta_t*} & (t^\vee \otimes t)* \ar[r] & t^\vee*t*}
\end{displaymath}
and
\begin{displaymath}
\e_t'\colon \xymatrix{t*t^\vee * \ar[r] & (t\otimes t^\vee)* \ar[r]^(0.6){\e_t*} & \mathbf{1}* \ar[r]^{l} & \id_\mathcal{K}}
\end{displaymath}
which we claim are the unit and counit of an adjunction between $t*$ and $t^\vee *$. In order to prove this it is sufficient to verify that the composites
\begin{displaymath}
\xymatrix{t* \ar[r]^(0.3){t*\eta_t'} & t*t^\vee*t* \ar[r]^(0.7){\e_t' t*} & t*} \quad \text{and} \quad \xymatrix{t^\vee * \ar[r]^(0.4){\eta_t' t^\vee *} & t^\vee*t*t^\vee \ar[r]^(0.6){t^\vee*\e_t'} & t^\vee*}
\end{displaymath}
are the respective identity natural transformations (see for instance \cite{MacLane} IV.1 Theorem 2). In fact these are precisely the identity composites corresponding to the existence of strong duals in $\mathcal{T}$ applied to $\mathcal{K}$. This is easily checked using the compatibility conditions required for $\mathcal{T}$ to act on $\mathcal{K}$.

Thus $\eta_t'$ and $\e_t'$ give the desired adjunction. In particular, $t*$ has a coproduct preserving right adjoint and so by \cite{NeeGrot} Theorem 5.1 it must send compact objects to compact objects.
\end{proof}

\begin{rem}\label{rem_adjoints}
It is worth noting that we proved more than we stated: for each $t\in \mathcal{T}^c$ the functor $t*$ has a right adjoint given via acting by another compact object namely, $t^\vee *$.
\end{rem}

Of course there are other situations in which this is true, although one has to assume more.

\begin{lem}
Let $\mathcal{T}$ be a (not necessarily rigidly) compactly generated tensor triangulated category acting on a compactly generated triangulated category $\mathcal{K}$. If there exists a set of compact generators $\{x_i\}_{i\in I}$ for $\mathcal{T}$ such that $x_i*\mathcal{K}^c \cie \mathcal{K}^c$ for each $i\in I$ then the action of $\mathcal{T}$ on $\mathcal{K}$ restricts to an action of $\mathcal{T}^c$ on $\mathcal{K}^c$. In particular, if the unit object $\mathbf{1}$ of $\mathcal{T}$ is compact and generates $\mathcal{T}$ the action restricts.
\end{lem}
\begin{proof}
The argument is standard - it follows from the obvious analogue of Lemma \ref{lem_general_closure} for thick subcategories.
\end{proof}

In such a situation the general results on generators we proved in Section \ref{sec_actions} allows us to produce localizing submodules of $\mathcal{K}$ generated by objects of $\mathcal{K}^c$. In particular it implies the subcategories of the form $\mathit{\Gamma}_\mathcal{V}\mathcal{K}$ for $\mathcal{V}$ a Thomason subset of $\Spc \mathcal{T}^c$ are generated by compact objects of $\mathcal{K}$.

\begin{prop}\label{prop_action_generation}
Suppose $\mathcal{T}$ acts on $\mathcal{K}$, with both $\mathcal{T}$ and $\mathcal{K}$ compactly generated, in such a way that the action restricts to one of $\mathcal{T}^c$ on $\mathcal{K}^c$ (e.g., $\mathcal{T}$ is rigidly-compactly generated). Then given a $\otimes$-ideal $\mathcal{L}\cie \mathcal{T}$ generated, as a localizing subcategory, by compact objects of $\mathcal{T}$ and a localizing subcategory $\mathcal{M}\cie \mathcal{K}$ generated by objects of $\mathcal{K}^c$ the subcategory $\mathcal{L}*\mathcal{M}$ is also generated, as a localizing subcategory, by compact objects of $\mathcal{K}$.
\end{prop}
\begin{proof}
Let us fix generating sets $\{c_\l\}_{\l \in \Lambda}$ for $\mathcal{T}$, $\{x_i\}_{i\in I}$ for $\mathcal{L}$, and $\{a_j\}_{j\in J}$ for $\mathcal{M}$  where the $c_\l$ and $x_i$ lie in $\mathcal{T}^c$ and the $a_j$ lie in $\mathcal{K}^c$. We have equalities of submodules
\begin{align*}
\mathcal{L}*\mathcal{M} = \langle x_i \; \vert \; i\in I\rangle * \langle a_j \; \vert \; j\in J\rangle = \langle Z*(x_i*a_j) \; \vert \; Z\in \mathcal{T}, i\in I, j\in J\rangle
\end{align*}
by Lemma \ref{lem_loc_commutes2}. Since $\mathcal{T} = \langle c_\l \; \vert \; \l\in \Lambda\rangle$  we can use Lemma \ref{lem_loc_commutes} to rewrite this as
\begin{displaymath}
\langle Z*(x_i*a_j) \; \vert \; Z\in \mathcal{T}, i\in I, j\in J\rangle = \langle c_\l*(x_i*a_j) \; \vert \; \l\in \Lambda, i\in I, j\in J\rangle
\end{displaymath}
which completes the proof as by hypothesis the action restricts to compacts.
\end{proof}

\begin{rem}\label{rem_action_generation}
We get more from the proof of this proposition when $\mathcal{T}$ is generated by the tensor unit. In this case all localizing and thick subcategories are submodules for $\mathcal{T}$ and $\mathcal{T}^c$ respectively so we do not need to close under the action. Thus, with this added assumption, we showed that if $\mathcal{L}$ is generated by objects $\{x_i\}_{i\in I}$ of $\mathcal{T}^c$ and $\mathcal{M}$ is generated by objects $\{a_j\}_{j\in J}$ of $\mathcal{K}^c$ then $\mathcal{L}*\mathcal{M}$ has a generating set $\{x_i*a_j\}_{i\in I,j\in J}$ of objects compact in $\mathcal{K}$.
\end{rem}

\begin{cor}\label{cor_transfer_cg}
Suppose $\mathcal{T}$ is a rigidly-compactly generated tensor triangulated category acting on a compactly generated triangulated category $\mathcal{K}$ and that $\mathcal{V}$ is a Thomason subset of $\Spc \mathcal{T}^c$. Then the subcategory
\begin{displaymath}
\mathit{\Gamma}_\mathcal{V}\mathcal{K} = \{A \in \mathcal{K} \; \vert \; \exists A' \; \text{with}\; A\iso \mathit{\Gamma}_\mathcal{V}\mathbf{1}*A' \}
\end{displaymath}
is generated by compact objects of $\mathcal{K}$.
\end{cor}
\begin{proof}
By the proposition we have just proved it is sufficient to make the identification $\mathit{\Gamma}_\mathcal{V}\mathcal{K} = \mathit{\Gamma}_\mathcal{V}\mathcal{T} * \mathcal{K}$. If $X$ is an object of $\mathit{\Gamma}_\mathcal{V}\mathcal{T}$ then there is an isomorphism $X\iso \mathit{\Gamma}_\mathcal{V}\mathbf{1}\otimes X$. Thus we have
\begin{align*}
\mathit{\Gamma}_\mathcal{V}\mathcal{T}*\mathcal{K} &= \langle X*A \; \vert \; X\in \mathit{\Gamma}_\mathcal{V}\mathcal{T}, A\in \mathcal{K}\rangle_* \\
&= \langle \mathit{\Gamma}_\mathcal{V}\mathbf{1}*(X*A) \; \vert \; X\in \mathit{\Gamma}_\mathcal{V}\mathcal{T}, A\in \mathcal{K}\rangle_* \\
&= \langle \mathit{\Gamma}_\mathcal{V}\mathbf{1}*A \; \vert \; A\in \mathcal{K}\rangle_*. \\
\end{align*}
Closing the generators of this last submodule under isomorphisms gives $\mathit{\Gamma}_\mathcal{V}\mathcal{K}$ which, by Lemma \ref{lem_action_transfer1}, is a localizing $\mathcal{T}$-submodule. Thus $\mathit{\Gamma}_\mathcal{V}\mathcal{K} = \mathit{\Gamma}_\mathcal{V}\mathcal{T}*\mathcal{K}$ and we can apply the last proposition to complete the proof.
\end{proof}

\section{Supports via actions}\label{sec_supports}

We now define the functors which give rise to supports on $\mathcal{K}$ relative to an action $(\mathcal{T},*)$. We assume that $\mathcal{K}$ is compactly generated and $\mathcal{T}$ is rigidly-compactly generated.

\begin{defn}\label{defn_gen_ptfunctors}
For every $x \in \Spc \mathcal{T}^c$ we define subsets of the spectrum
\begin{displaymath}
\mathcal{V}(x) = \overline{\{x\}}
\end{displaymath}
and
\begin{displaymath}
\mathcal{Z}(x) = \{y\in \Spc \mathcal{T}^c \; \vert \; x\notin \mathcal{V}(y)\}.
\end{displaymath}
Both of these subsets are specialization closed but $\mathcal{V}(x)$ may fail to be Thomason (the subset $\mathcal{Z}(x)$ is always Thomason by virtue of being the support of the prime ideal giving the point $x$).

\begin{ex}
Let us give an example illustrating the failure of a closed point to be Thomason. Consider the linearly ordered set $\nat \un \{\infty\}$ equipped with the specialization topology i.e., for $x,y\in \nat \un \{\infty\}$ we have $x\in \mathcal{V}(y)$ if and only if $x\geq y$. One sees that the closed subset $\mathcal{V}(\infty) = \{\infty\}$ is not Thomason; the open cover
\begin{displaymath}
\nat = \bigcup_{n\in \nat} \{m\in \nat \; \vert \; m\leq n\}
\end{displaymath}
does not admit a finite refinement. It is worth remarking that this example occurs in nature as the spectrum of the $p$-local stable homotopy category (see \cite{BaSSS} Corollary~9.5).

In fact, it is shown in \cite{BaRickard} Proposition 7.13 that as soon as the spectrum of $\mathcal{T}^c$ is not noetherian there exists a point whose closure is not Thomason.
\end{ex}

\begin{defn}\label{defn_gen_ptfunctors2}
Let $x$ be a point of $\Spc \mathcal{T}^c$. In the case that $\mathcal{V}(x)$ is Thomason we define a $\otimes$-idempotent
\begin{displaymath}
\mathit{\Gamma}_x\mathbf{1}= (\mathit{\Gamma}_{\mathcal{V}(x)}\mathbf{1} \otimes L_{\mathcal{Z}(x)}\mathbf{1}).
\end{displaymath}
\end{defn}

In keeping with previous notation we will sometimes write $\mathit{\Gamma}_xA$ instead of $\mathit{\Gamma}_x\mathbf{1}*A$ for objects $A$ of $\mathcal{K}$. We recall from \cite{BaRickard} Corollary 7.5 that the idempotent functors $\mathit{\Gamma}_x\mathbf{1}\otimes(-)$ on $\mathcal{T}$ for $x\in \Spc \mathcal{T}^c$ only depend on $x$. If one makes a different choice of Thomason subsets $\mathcal{W},\mathcal{V}$ satisfying $\mathcal{V}\setminus \{\mathcal{V}\intersec \mathcal{W}\} = \{x\}$ then $\mathit{\Gamma}_\mathcal{V}\mathbf{1}\otimes L_\mathcal{W}\mathbf{1}$ is naturally isomorphic to $\mathit{\Gamma}_x\mathbf{1}$ (cf.\ Theorem 6.2 of \cite{BIK}). Thus, with $\mathcal{T}$ acting on $\mathcal{K}$, the functors $\mathit{\Gamma}_x\colon \mathcal{K} \to \mathcal{K}$ also only depend on $x$. In other words we have:

\begin{lem}\label{lem_bik6.2_uber}
Let $x\in \Spc \mathcal{T}^c$ and suppose $\mathcal{V}$ and $\mathcal{W}$ are Thomason subsets of $\Spc \mathcal{T}^c$ such that $\mathcal{V} \setminus (\mathcal{V}\intersec \mathcal{W}) = \{x\}$. Then there are natural isomorphisms
\begin{displaymath}
(L_\mathcal{W}\mathbf{1} \otimes \mathit{\Gamma}_\mathcal{V}\mathbf{1})*(-) \iso \mathit{\Gamma}_x \iso (\mathit{\Gamma}_\mathcal{V}\mathbf{1} \otimes L_\mathcal{W}\mathbf{1})*(-).
\end{displaymath}
\end{lem}

If such sets exist for $x\in \Spc \mathcal{T}^c$ let us follow the terminology of \cite{BaRickard} and call $x$ \emph{visible}. By \cite{BaRickard} Corollary 7.14 every point is visible in our sense if the spectrum of $\mathcal{T}^c$ is noetherian. We denote by $\Vis\mathcal{T}^c$ the subspace of visible points of $\mathcal{T}$.
\end{defn}

\begin{notation}
Following previous notation we use $\mathit{\Gamma}_x\mathcal{K}$, for $x\in \Spc \mathcal{T}^c$, to denote the essential image of $\mathit{\Gamma}_x\mathbf{1}*(-)$. It is a $\mathcal{T}$-submodule as for any $X \in \mathcal{T}$ and $A\in \mathit{\Gamma}_x \mathcal{K}$
\begin{align*}
X*A \iso X*(\mathit{\Gamma}_x\mathbf{1}*A') \iso \mathit{\Gamma}_x\mathbf{1}*(X*A')
\end{align*}
for some $A'\in \mathcal{K}$.
\end{notation}

We can define supports taking values in the set of visible points of $\Spc \mathcal{T}^c$.
\begin{defn}
Given $A$ in $\mathcal{K}$ we define the support of $A$ to be the set
\begin{displaymath}
\supp_{(\mathcal{T},*)} A = \{x\in \Vis \mathcal{T}^c \; \vert \; \mathit{\Gamma}_x A \neq 0\}.
\end{displaymath}
When the action in question is clear we will omit the subscript from the notation.
\end{defn}

The following proposition, which gives the basic properties of the support assignment, already appears in \cite{BaRickard} (more precisely see Propositions 7.17 and 7.18) in the case $\mathcal{T}$ acts on itself. However, we include a proof of (4) both for completeness and to reinforce that one only uses formal properties of the Rickard idempotents.

\begin{prop}\label{prop_abs_supp_prop}
The support assignment $\supp_{(\mathcal{T},*)}$ satisfies the following properties:
\begin{itemize}
\item[$(1)$] given a triangle
\begin{displaymath}
A \to B \to C \to \S A
\end{displaymath}
in $\mathcal{K}$ we have $\supp B \cie \supp A \un \supp C$;
\item[$(2)$] for any $A$ in $\mathcal{K}$ and $i\in \int$
\begin{displaymath}
\supp A = \supp \S^i A;
\end{displaymath}
\item[$(3)$] given a set-indexed family $\{A_\l\}_{\l \in \Lambda}$ of objects of $\mathcal{K}$ there is an equality
\begin{displaymath}
\supp \coprod_\l A_\l = \Un_\l \supp A_\l;
\end{displaymath}
\item[$(4)$] the support satisfies the separation axiom  i.e., for every specialization closed subset $\mathcal{V} \cie \Vis \mathcal{T}^c$ and every object $A$ of $\mathcal{K}$
\begin{align*}
&\supp \mathit{\Gamma}_\mathcal{V}\mathbf{1}* A = (\supp A) \intersec \mathcal{V} \\
&\supp L_\mathcal{V}\mathbf{1}*A = (\supp A) \intersec (\Vis \mathcal{T}^c \setminus \mathcal{V}).
\end{align*}
\end{itemize}
\end{prop}
\begin{proof}
As $\mathit{\Gamma}_x\mathbf{1}*(-)$ is a coproduct preserving exact functor (1), (2), and (3) are immediate. To see the separation axiom holds suppose $\mathcal{V}\cie \Vis \mathcal{T}^c$ is a specialization closed subset and let $A$ be an object of $\mathcal{K}$. Then 
\begin{align*}
\mathit{\Gamma}_x\mathbf{1}*(\mathit{\Gamma}_\mathcal{V}\mathbf{1}*A) &\iso (\mathit{\Gamma}_x\mathbf{1} \otimes \mathit{\Gamma}_\mathcal{V}\mathbf{1}) *A \\
&= (\mathit{\Gamma}_\mathcal{W}\mathbf{1} \otimes L_\mathcal{Y}\mathbf{1} \otimes \mathit{\Gamma}_\mathcal{V}\mathbf{1})*A
\end{align*}
where $\mathcal{W}$ and $\mathcal{Y}$ are Thomason subsets such that $\mathcal{W}\setminus (\mathcal{W}\intersec \mathcal{Y}) = \{x\}$. If $x\in \mathcal{V}$ the subsets $\mathcal{W}\intersec \mathcal{V}$ and $\mathcal{Y}$ also satisfy the conditions of Lemma \ref{lem_bik6.2_uber} i.e.,
\begin{displaymath}
(\mathcal{W}\intersec\mathcal{V}) \setminus (\mathcal{W}\intersec\mathcal{V}\intersec\mathcal{Y}) = \{x\}.
\end{displaymath} 
By \cite{BaRickard} Proposition 3.11 $\mathit{\Gamma}_\mathcal{W}\mathbf{1} \otimes \mathit{\Gamma}_\mathcal{V}\mathbf{1} \iso \mathit{\Gamma}_{\mathcal{W}\intersec \mathcal{V}}\mathbf{1}$. So in this case
\begin{displaymath}
\mathit{\Gamma}_x\mathbf{1}*\mathit{\Gamma}_\mathcal{V}\mathbf{1}*A \iso (\mathit{\Gamma}_{\mathcal{W}\intersec \mathcal{V}}\mathbf{1} \otimes L_\mathcal{Y}\mathbf{1}) *A \iso \mathit{\Gamma}_x\mathbf{1}*A.
\end{displaymath}
If $x\notin \mathcal{V}$ then $\mathcal{W}\intersec\mathcal{V}$ is contained in $\mathcal{Y}$. It follows that $\mathit{\Gamma}_{\mathcal{W}\intersec\mathcal{V}}\mathcal{T} \cie \mathit{\Gamma}_\mathcal{Y}\mathcal{T}$ so, using standard facts about acyclization and localization functors e.g.\ \cite{BIK} Lemma 3.4,
\begin{displaymath}
\mathit{\Gamma}_x\mathbf{1}*\mathit{\Gamma}_\mathcal{V}\mathbf{1}*A \iso 0.
\end{displaymath}
This proves $\supp \mathit{\Gamma}_\mathcal{V}\mathbf{1}*A = (\supp A) \intersec \mathcal{V}$. One proves the analogue for $L_\mathcal{V}\mathbf{1}*A$ similarly.
\end{proof}

\begin{cor}\label{cor_supp_point1}
Let $x$ be a visible point of $\Spc \mathcal{T}^c$. Then, for $\mathcal{T}$ acting on itself, $\supp \mathit{\Gamma}_x\mathbf{1} = \{x\}$. We also have that for distinct points $x_1,x_2$ of $\Vis \mathcal{T}^c$ the tensor product $\mathit{\Gamma}_{x_1}\mathbf{1}\otimes \mathit{\Gamma}_{x_2}\mathbf{1}$ vanishes.
\end{cor}
\begin{proof}
Let $\mathcal{V}$ and $\mathcal{W}$ be Thomason subsets giving rise to $\mathit{\Gamma}_x\mathbf{1}$. Statement (4) of the proposition implies
\begin{align*}
\supp \mathit{\Gamma}_x \mathbf{1} & = \supp (\mathit{\Gamma}_\mathcal{V}\mathbf{1} \otimes (L_\mathcal{W}\mathbf{1} \otimes \mathbf{1})) \\
& = \mathcal{V} \intersec \supp (L_\mathcal{W}\mathbf{1}\otimes \mathbf{1}) \\
&= \mathcal{V} \intersec (\Vis \mathcal{T}^c \setminus \mathcal{W}) \intersec \supp \mathbf{1} \\
&= \mathcal{V} \intersec (\Vis \mathcal{T}^c \setminus \mathcal{W}) \intersec \Vis \mathcal{T}^c \\
&= \{x\}
\end{align*}
which proves the first part of the corollary.

For the second statement recall from \cite{BaRickard} Remark 7.6 that $\mathit{\Gamma}_{x_1}\mathbf{1}\otimes \mathit{\Gamma}_{x_2}\mathbf{1}$ is isomorphic to $\mathit{\Gamma}_\varnothing\mathbf{1}$. Given any Thomason subset $\mathcal{V}$ we have
\begin{displaymath}
\mathit{\Gamma}_\varnothing\mathbf{1} \iso \mathit{\Gamma}_\mathcal{V}\mathbf{1} \otimes L_\mathcal{V}\mathbf{1} \iso 0,
\end{displaymath}
by \cite{BaRickard} Corollary 7.5, which shows the tensor product in question vanishes as claimed.
\end{proof}

Finally we can in this generality define a pair of assignments between visible subsets of $\Spc \mathcal{T}^c$ and localizing submodules of $\mathcal{K}$.

\begin{defn}\label{defn_vis_sigmatau}
We say a subset $W\cie \Spc \mathcal{T}^c$ is \emph{visible} if every $x\in W$ is a visible point or equivalently if $W\cie \Vis \mathcal{T}^c$. There are order preserving assignments
\begin{displaymath}
\left\{ \begin{array}{c}
\text{visible} \\ \text{subsets of}\; \Spc\mathcal{T}^c 
\end{array} \right\}
\xymatrix{ \ar[r]<1ex>^\t \ar@{<-}[r]<-1ex>_\s &} \left\{
\begin{array}{c}
\text{localizing submodules} \; \text{of} \; \mathcal{K} \\
\end{array} \right\} 
\end{displaymath}
where both collections are ordered by inclusion. For a localizing submodule $\mathcal{L}$ we set
\begin{displaymath}
\s(\mathcal{L}) = \supp \mathcal{L} = \{x \in \Vis\mathcal{T}^c \; \vert \; \mathit{\Gamma}_x\mathcal{L} \neq 0\}
\end{displaymath}
and 
\begin{displaymath}
\t(W) = \{A \in \mathcal{K} \; \vert \; \supp A \cie W\}.
\end{displaymath}
Both of these are well defined; this is clear for $\s$ and for $\t$ it follows from Proposition \ref{prop_abs_supp_prop}.
\end{defn}

\section[The local-to-global principle]{Homotopy colimits and the local-to-global principle}\label{ssec_ltg}
Throughout this section we fix an action $\mathcal{T}\times\mathcal{K} \stackrel{*}{\to} \mathcal{K}$ where $\mathcal{T}$ is a rigidly-compactly generated tensor triangulated category and $\mathcal{K}$ is compactly generated. Furthermore, we assume $\Spc \mathcal{T}^c$ is a noetherian topological space so that specialization closed subsets are the same as Thomason subsets and all points are visible. All submodules are again assumed to be localizing.

We begin by generalizing the local-to-global principle of \cite{BIKstrat2}.

\begin{defn}\label{defn_ltg}
We say $\mathcal{T}\times\mathcal{K} \stackrel{*}{\to} \mathcal{K}$ satisfies the \emph{local-to-global principle} if for each $A$ in $\mathcal{K}$
\begin{displaymath}
\langle A \rangle_* = \langle \mathit{\Gamma}_x A \; \vert \; x\in \Spc \mathcal{T}^c\rangle_*.
\end{displaymath}
\end{defn}

The local-to-global principle has the following rather pleasing consequences for the assignments $\s$ and $\t$ of Definition \ref{defn_vis_sigmatau}.

\begin{lem}\label{general_im_tau}
Suppose the local-to-global principle holds for the action of $\mathcal{T}$ on $\mathcal{K}$ and let $W$ be a subset of $\Spc \mathcal{T}^c$. Then
\begin{displaymath}
\t(W) = \langle \mathit{\Gamma}_x\mathcal{K} \; \vert \; x\in W\intersec \s\mathcal{K} \rangle_*.
\end{displaymath}
\end{lem}
\begin{proof}
By the local-to-global principle we have for every object $A$ of $\mathcal{K}$ an equality
\begin{displaymath}
\langle A \rangle_* = \langle \mathit{\Gamma}_xA \; \vert \; x\in \Spc\mathcal{T}^c\rangle_*.
\end{displaymath}
Thus
\begin{align*}
\t(W) &= \langle A \; \vert \; \supp A \cie W\rangle_* \\
&= \langle \mathit{\Gamma}_x A \; \vert \; A\in \mathcal{K}, x\in W\rangle_* \\
&= \langle \mathit{\Gamma}_x A \; \vert \; A\in \mathcal{K}, x\in W\intersec \s\mathcal{K}\rangle_* \\
&= \langle \mathit{\Gamma}_x\mathcal{K} \; \vert \; x\in W\intersec \s\mathcal{K} \rangle_*.
\end{align*}
\end{proof}

\begin{prop}\label{general_tau_inj}
Suppose the local-to-global principle holds for the action of $\mathcal{T}$ on $\mathcal{K}$ and let $W$ be a subset of $\Spc \mathcal{T}^c$. Then there is an equality of subsets
\begin{displaymath}
\s\t(W) = W \intersec \s\mathcal{K}.
\end{displaymath}
In particular, $\t$ is injective when restricted to subsets of $\s\mathcal{K}$.
\end{prop}
\begin{proof}
With $W\cie \Spc \mathcal{T}^c$ as in the statement we have
\begin{align*}
\s\t(W) &= \supp \t(W) \\
&= \supp \langle \mathit{\Gamma}_x\mathcal{K} \; \vert \; x\in W\intersec \s\mathcal{K} \rangle_*,
\end{align*}
the first equality by definition and the second by the last lemma. Thus $\s\t(W) = W\intersec \s\mathcal{K}$ as claimed: by the properties of the support (Proposition \ref{prop_abs_supp_prop}) we have $\s\t(W) \cie W\intersec \s\mathcal{K}$ and it must in fact be all of $W\intersec \s\mathcal{K}$ as $x\in \s\mathcal{K}$ if and only if $\mathit{\Gamma}_x\mathcal{K}$ contains a non-zero object.

\end{proof}

We will show that the local-to-global principle holds quite generally. Before proceeding let us fix some terminology we will use throughout the paper.

\begin{defn}\label{defn_hocolim}
We will say the tensor triangulated category $\mathcal{T}$ \emph{has a model} if it occurs as the homotopy category of a monoidal model category.
\end{defn}

Our main interest in such categories is that the existence of a monoidal model provides a good theory of homotopy colimits compatible with the tensor product. 

\begin{rem}
Of course instead of requiring that $\mathcal{T}$ arose from a monoidal model category we could, for instance, ask that $\mathcal{T}$ was the underlying category of a stable monoidal derivator. In fact we will only use directed homotopy colimits so one could use a weaker notion of a stable monoidal ``derivator'' only having homotopy left and right Kan extensions for certain diagrams; to be slightly more precise one could just ask for homotopy left and right Kan extensions along the smallest full $2$-subcategory of the category of small categories satisfying certain natural closure conditions and containing the ordinals (one can see the discussion before \cite{GrothPSD} Definition 4.21 for further details).
\end{rem}

We begin by showing that, when $\mathcal{T}$ has a model, taking the union of a chain of specialization closed subsets is compatible with taking the homotopy colimit of the associated idempotents.

\begin{lem}\label{lem_hocolim_ltg}
Suppose $\mathcal{T}$ has a model. Then for any chain $\{\mathcal{V}_i\}_{i\in I}$ of specialization closed subsets of $\Spc \mathcal{T}^c$ with union $\mathcal{V}$ there is an isomorphism
\begin{displaymath}
\mathit{\Gamma}_\mathcal{V}\mathbf{1} \iso \hocolim \mathit{\Gamma}_{\mathcal{V}_i}\mathbf{1}
\end{displaymath}
where the structure maps are the canonical ones.
\end{lem}
\begin{proof}
As each $\mathcal{V}_i$ is contained in $\mathcal{V}$ there are corresponding inclusions for $i<j$ 
\begin{displaymath}
\mathcal{T}_{\mathcal{V}_i} \cie \mathcal{T}_{\mathcal{V}_j} \cie \mathcal{T}_\mathcal{V}
\end{displaymath}
which give rise to commuting triangles of canonical morphisms 
\begin{displaymath}
\xymatrix{
\mathit{\Gamma}_{\mathcal{V}_i}\mathbf{1} \ar[rr] \ar[dr] && \mathit{\Gamma}_{\mathcal{V}}\mathbf{1} \\
& \mathit{\Gamma}_{\mathcal{V}_j}\mathbf{1} \ar[ur]
}
\end{displaymath}
We thus get an induced morphism from the homotopy colimit of the $\mathit{\Gamma}_{\mathcal{V}_i}\mathbf{1}$ to $\mathit{\Gamma}_\mathcal{V}\mathbf{1}$ which we complete to a triangle
\begin{displaymath}
\hocolim_I \mathit{\Gamma}_{\mathcal{V}_i}\mathbf{1} \to \mathit{\Gamma}_\mathcal{V}\mathbf{1} \to Z \to \S \hocolim_I \mathit{\Gamma}_{\mathcal{V}_i}\mathbf{1}.
\end{displaymath}
In order to prove the lemma it is sufficient to show that $Z$ is isomorphic to the zero object in $\mathcal{T}$.

The argument in \cite{BousfieldBAS} extends to show localizing subcategories are closed under homotopy colimits so this triangle consists of objects of  $\mathit{\Gamma}_\mathcal{V}\mathcal{T}$. By definition $\mathit{\Gamma}_\mathcal{V}\mathcal{T}$ is the full subcategory of $\mathcal{T}$ generated by those objects of $\mathcal{T}^c$ whose support (in the sense of \cite{BaSpec}) is contained in $\mathcal{V}$. Thus $Z\iso 0$ if for each compact object $k$ with $\supp k \cie \mathcal{V}$ we have $\Hom(k,Z) = 0$; we remark that there is no ambiguity here as by \cite{BaRickard} Proposition 7.17 the two notions of support, that of \cite{BaSpec} and \cite{BaRickard}, agree for compact objects. In particular the support of any compact object is closed.

Recall from \cite{BKS} that $\Spc \mathcal{T}^c$ is spectral in the sense of Hochster \cite{HochsterSpectral} and we have assumed it is also noetherian. Thus $\supp k$, by virtue of being closed, is a finite union of irreducible closed subsets. We can certainly find a $j\in I$ so that $\mathcal{V}_j$ contains the generic points of these finitely many irreducible components which implies $\supp k \cie \mathcal{V}_j$ by specialization closure. Therefore, by adjunction, it is enough to show 
\begin{align*}
\Hom(k, Z) &\iso \Hom(\mathit{\Gamma}_{\mathcal{V}_j} k ,Z) \\
&\iso \Hom(k, \mathit{\Gamma}_{\mathcal{V}_j} Z)
\end{align*}
is zero, as this implies $Z \iso 0$ and we get the claimed isomorphism.

In order to show the claimed hom-set vanishes let us demonstrate that $\mathit{\Gamma}_{\mathcal{V}_j}Z$ is zero. Observe that tensoring the structure morphisms $\mathit{\Gamma}_{\mathcal{V}_{i_1}}\mathbf{1} \to \mathit{\Gamma}_{\mathcal{V}_{i_2}}\mathbf{1}$ for $i_2\geq i_1\geq j$ with $\mathit{\Gamma}_{\mathcal{V}_j}\mathbf{1}$ yields canonical isomorphisms
\begin{displaymath}
\mathit{\Gamma}_{\mathcal{V}_j}\mathbf{1} \iso \mathit{\Gamma}_{\mathcal{V}_j}\mathbf{1}\otimes \mathit{\Gamma}_{\mathcal{V}_{i_1}}\mathbf{1} \stackrel{\sim}{\to} \mathit{\Gamma}_{\mathcal{V}_j}\mathbf{1}\otimes \mathit{\Gamma}_{\mathcal{V}_{i_2}}\mathbf{1} \iso \mathit{\Gamma}_{\mathcal{V}_j}\mathbf{1}.
\end{displaymath}
Thus applying $\mathit{\Gamma}_{\mathcal{V}_j}$ to the sequence $\{\mathit{\Gamma}_{\mathcal{V}_i}\mathbf{1}\}_{i\in I}$ gives a diagram whose homotopy colimit is $\mathit{\Gamma}_{\mathcal{V}_j}\mathbf{1}$. From this we deduce that the first morphism in the resulting triangle
\begin{displaymath}
\mathit{\Gamma}_{\mathcal{V}_j}\mathbf{1} \otimes \hocolim_I \mathit{\Gamma}_{\mathcal{V}_i}\mathbf{1} \to \mathit{\Gamma}_{\mathcal{V}_j}\mathbf{1} \to \mathit{\Gamma}_{\mathcal{V}_j}Z 
\end{displaymath}
is an isomorphism as $\mathcal{T}$ is the homotopy category of a monoidal model category so the tensor product commutes with homotopy colimits. This forces $\mathit{\Gamma}_{\mathcal{V}_j}Z \iso 0$ completing the proof.
\end{proof}

\begin{lem}\label{lem_Amnon2.10_uber}
Let $P\cie \Spc \mathcal{T}^c$ be given and suppose $A$ is an object of $\mathcal{K}$ such that $\mathit{\Gamma}_x A \iso 0$ for all $x\in (\Spc \mathcal{T}^c \setminus P)$. If $\mathcal{T}$ has a model then $A$ is an object of the localizing subcategory
\begin{displaymath}
\mathcal{L} = \langle \mathit{\Gamma}_y\mathcal{K} \; \vert \; y\in P\rangle_{\mathrm{loc}}.
\end{displaymath}
\end{lem}
\begin{proof}
Let $\Lambda \cie \mathcal{P}(\Spc \mathcal{T}^c)$ be the set of specialization closed subsets $\mathcal{W}$ such that $\mathit{\Gamma}_\mathcal{W}A$ is in $\mathcal{L} = \langle \mathit{\Gamma}_y\mathcal{K} \; \vert \; y\in P\rangle_{\mathrm{loc}}$. We first note that $\Lambda$ is not empty. Indeed, as $\mathcal{T}^c$ is rigid the only compact objects with empty support are the zero objects by \cite{BaFilt} Corollary 2.5 so
\begin{displaymath}
\mathcal{T}_\varnothing = \langle t\in \mathcal{T}^c \; \vert \; \supp_{(\mathcal{T},\otimes)} t = \varnothing \rangle_\mathrm{loc} = \langle 0 \rangle_\mathrm{loc}
\end{displaymath}
giving $\mathit{\Gamma}_\varnothing A = 0$ and hence $\varnothing \in \Lambda$. 

Since $\mathcal{L}$ is localizing, Lemma \ref{lem_hocolim_ltg} shows the set $\Lambda$ is closed under taking increasing unions: as mentioned above the argument in \cite{BousfieldBAS} extends to show that localizing subcategories are closed under directed homotopy colimits in our situation. Thus $\Lambda$ contains a maximal element $Y$ by Zorn's lemma. We claim that $Y = \Spc \mathcal{T}^c$.

Suppose $Y\neq \Spc \mathcal{T}^c$. Then since $\Spc \mathcal{T}^c$ is noetherian $\Spc\mathcal{T}^c \setminus Y$ contains a maximal element $z$ with respect to specialization. We have 
\begin{displaymath}
L_Y\mathbf{1}\otimes\mathit{\Gamma}_{Y\un \{z\}}\mathbf{1} \iso \mathit{\Gamma}_z\mathbf{1}
\end{displaymath}
as $Y\un \{z\}$ is specialization closed by maximality of $z$ and Lemma \ref{lem_bik6.2_uber} tells us that we can use any suitable pair of Thomason subsets to define $\mathit{\Gamma}_z\mathbf{1}$. So $L_Y\mathit{\Gamma}_{Y\un \{z\}}A \iso \mathit{\Gamma}_z A$ and by our hypothesis on vanishing either $\mathit{\Gamma}_z \mathcal{K} \cie \mathcal{L}$ or \mbox{$\mathit{\Gamma}_z A = 0$}. Considering the triangle
\begin{displaymath}
\xymatrix{
\mathit{\Gamma}_Y\mathit{\Gamma}_{Y\un \{z\}} A \ar[d]_{\wr} \ar[r] & \mathit{\Gamma}_{Y\un \{z\}}A \ar[r] & L_Y\mathit{\Gamma}_{Y\un \{z\}}A \ar[d]^{\wr} \\
\mathit{\Gamma}_Y A & & \mathit{\Gamma}_z A
}
\end{displaymath}
we see that in either case, since $\mathit{\Gamma}_YA$ is in $\mathcal{L}$, that $Y\un \{z\}\in \Lambda$ contradicting maximality of $Y$. Hence $Y = \Spc \mathcal{T}^c$ and so $A$ is in $\mathcal{L}$.
\end{proof}

\begin{prop}\label{prop_ltg_idempotents}
Suppose $\mathcal{T}$ has a model. Then the local-to-global principle holds for the action of $\mathcal{T}$ on $\mathcal{K}$. Explicitly for any $A$ in $\mathcal{K}$ there is an equality of $\mathcal{T}$-submodules
\begin{displaymath}
\langle A \rangle_* = \langle \mathit{\Gamma}_x A \; \vert \; x\in \supp A \rangle_*.
\end{displaymath}
\end{prop}
\begin{proof}
By Lemma \ref{lem_Amnon2.10_uber} applied to the action
\begin{displaymath}
\mathcal{T}\times \mathcal{T} \stackrel{\otimes}{\to} \mathcal{T}
\end{displaymath}
we see $\mathcal{T} = \langle \mathit{\Gamma}_x\mathcal{T}\; \vert \; x\in \Spc \mathcal{T}^c\rangle_\mathrm{loc}$. Since $\mathit{\Gamma}_x\mathcal{T} = \langle \mathit{\Gamma}_x\mathbf{1}\rangle_\otimes$ it follows that the set of objects $\{\mathit{\Gamma}_x\mathbf{1} \; \vert \; x\in \Spc \mathcal{T}^c \}$ generates $\mathcal{T}$ as a localizing $\otimes$-ideal. By Lemma \ref{lem_loc_commutes2} given an object $A\in \mathcal{K}$ we get a generating set for  $\mathcal{T}*\langle A\rangle_\mathrm{loc}$:
\begin{displaymath}
\mathcal{T}* \langle A \rangle_{\mathrm{loc}} = \langle \mathit{\Gamma}_x\mathbf{1} \; \vert \; x\in \Spc \mathcal{T}\rangle_\otimes * \langle A \rangle_{\mathrm{loc}}= \langle \mathit{\Gamma}_x A \; \vert \; x\in \supp A\rangle_*.
\end{displaymath}
But it is also clear that $\mathcal{T} = \langle \mathbf{1} \rangle_\otimes$ so, by Lemma \ref{lem_loc_commutes2} again,
\begin{displaymath}
\mathcal{T}* \langle A \rangle_\mathrm{loc} = \langle \mathbf{1} \rangle_\otimes * \langle A \rangle_\mathrm{loc} = \langle A \rangle_*
\end{displaymath}
and combining this with the other string of equalities gives
\begin{displaymath}
\langle A \rangle_* = \mathcal{T}*\langle A \rangle_\mathrm{loc} = \langle \mathit{\Gamma}_x A \; \vert \; x\in \supp A\rangle_*
\end{displaymath}
which completes the proof.

\end{proof}


We thus have the following theorem concerning the local-to-global principle for actions of rigidly-compactly generated tensor triangulated categories.

\begin{thm}\label{thm_general_ltg}
Suppose $\mathcal{T}$ is a rigidly-compactly generated tensor triangulated category with a model and that $\Spc \mathcal{T}^c$ is noetherian. Then $\mathcal{T}$ satisfies the following properties:
\begin{itemize}
\item[$(i)$] The local-to-global principle holds for the action of $\mathcal{T}$ on itself;
\item[$(ii)$] The associated support theory detects vanishing of objects i.e., $X \in \mathcal{T}$ is zero if and only if $\supp X = \varnothing$;
\item[$(iii)$] For any chain $\{\mathcal{V}_i\}_{i\in I}$ of specialization closed subsets of $\Spc \mathcal{T}^c$ with union $\mathcal{V}$ there is an isomorphism
\begin{displaymath}
\mathit{\Gamma}_\mathcal{V}\mathbf{1} \iso \hocolim \mathit{\Gamma}_{\mathcal{V}_i}\mathbf{1}
\end{displaymath}
where the structure maps are the canonical ones.
\end{itemize}
Furthermore, the relative versions of (i) and (ii) hold for any action of $\mathcal{T}$ on a compactly generated triangulated category $\mathcal{K}$.
\end{thm}
\begin{proof}
That (iii) always holds is the content of Lemma \ref{lem_hocolim_ltg} and we have proved in Proposition \ref{prop_ltg_idempotents} that (i) holds.
To see (i) implies (ii) observe that if $\supp X = \varnothing$ for an object $X$ of $\mathcal{T}$ then the local-to-global principle yields
\begin{displaymath}
\langle X \rangle_\otimes = \langle \mathit{\Gamma}_x X \; \vert \; x\in \Spc\mathcal{T}^c \rangle_\otimes = \langle 0 \rangle_\otimes
\end{displaymath}
so $X\iso 0$.

Finally, we saw in Proposition \ref{prop_ltg_idempotents} that the relative version of (i) holds. This in turn implies (ii) for supports with values in $\Spc \mathcal{T}^c$ by the same argument as we have used in the proof of (i)$\rimp$(ii) above.
\end{proof}

\section{The telescope conjecture}\label{sec_tele}

We now explore a relative version of the telescope conjecture. We show that for particularly nice actions $\mathcal{T}\times \mathcal{K} \stackrel{*}{\to} \mathcal{K}$ we can deduce the relative telescope conjecture for $\mathcal{K}$. We will denote by $\mathcal{T}$ a rigidly-compactly generated tensor triangulated category with noetherian spectrum (although let us note that some of the results hold more generally) and by $\mathcal{K}$ a compactly generated triangulated category on which $\mathcal{T}$ acts.

\begin{defn}\label{defn_relative_tele}
We say the \emph{relative telescope conjecture} holds for $\mathcal{K}$ with respect to the action of $\mathcal{T}$ if every smashing $\mathcal{T}$-submodule $\mathcal{S} \cie \mathcal{K}$ (we recall this means $\mathcal{S}$ is a localizing submodule with an associated coproduct preserving localization functor) is generated as a localizing subcategory by compact objects of $\mathcal{K}$. 
\end{defn}

\begin{rem}\label{rem_tele}
This reduces to the usual telescope conjecture if every localizing subcategory of $\mathcal{K}$ is a submodule, for example if $\mathcal{T}$ is generated as a localizing subcategory by $\mathbf{1}$. 
\end{rem}

\begin{lem}\label{lem_smashing_orthogideal}
Suppose $\mathcal{S} \cie \mathcal{K}$ is a smashing $\mathcal{T}$-submodule. Then $\mathcal{S}^\perp$ is a localizing $\mathcal{T}$-submodule.
\end{lem}
\begin{proof}
Let us denote by $\mathcal{L}$ the subcategory of those objects of $\mathcal{T}$ which send $\mathcal{S}^\perp$ to itself
\begin{displaymath}
\mathcal{L} = \{X\in \mathcal{T} \; \vert \; X*\mathcal{S}^\perp \cie \mathcal{S}^\perp\}.
\end{displaymath}
As $\mathcal{S}$ is smashing the subcategory $\mathcal{S}^\perp$ is a localizing subcategory of $\mathcal{K}$ (see for example \cite{KrLoc} Proposition 5.5.1). Thus $\mathcal{L}$ is a localizing subcategory of $\mathcal{T}$ by the standard argument.

If $x$ is a compact object of $\mathcal{T}$ then, as we have assumed $\mathcal{T}$ rigidly-compactly generated, the object $x$ is strongly dualizable. By Remark \ref{rem_adjoints} the functor $x*(-)$ has a right adjoint $x^\vee *(-)$ so
given $B$ in $\mathcal{S}^\perp$ we have, for every $A$ in $\mathcal{S}$,
\begin{displaymath}
0 =\Hom(x*A,B) \iso \Hom(A,x^\vee * B),
\end{displaymath}
where the first hom-set vanishes due to the fact that $\mathcal{S}$ is a submodule so $x*A$ is an object of $\mathcal{S}$. Hence $x^\vee *B$ is an object of $\mathcal{S}^\perp$ for every $x$ in $\mathcal{T}^c$. As taking duals of compact objects in $\mathcal{T}$ is involutive this implies that every object of $\mathcal{T}^c$ sends $\mathcal{S}^\perp$ to $\mathcal{S}^\perp$. Thus $\mathcal{T}^c$ is contained in the localizing subcategory $\mathcal{L}$ yielding the equality $\mathcal{L} = \mathcal{T}$. Hence every object $X$ of $\mathcal{T}$ satisfies $X*\mathcal{S}^\perp \cie \mathcal{S}^\perp$ so that $\mathcal{S}^\perp$ is a localizing $\mathcal{T}$-submodule of $\mathcal{K}$.
\end{proof}

\begin{defn}
Let $\mathcal{M}$ be a localizing $\mathcal{T}$-submodule of $\mathcal{K}$. We define a subcategory $\mathcal{T}_\mathcal{M}$ of $\mathcal{T}$ by
\begin{displaymath}
\mathcal{T}_\mathcal{M} = \{X \in \mathcal{T} \; \vert \; X*\mathcal{K} \cie \mathcal{M}\}.
\end{displaymath}
\end{defn}

\begin{lem}\label{lem_telecat_loc1}
Suppose $\mathcal{M}$ is a localizing submodule of $\mathcal{K}$. Then the subcategory $\mathcal{T}_\mathcal{M}$ is a localizing $\otimes$-ideal of $\mathcal{T}$. 
\end{lem}
\begin{proof}
Lemma \ref{lem_general_closure} tells us that $\mathcal{T}_\mathcal{M}$ is localizing. It is also easily seen that $\mathcal{T}_\mathcal{M}$ is a $\otimes$-ideal. If $X$ is an object of $\mathcal{T}_\mathcal{M}$, $Y$ is any object of $\mathcal{T}$, and $A$ is in $\mathcal{K}$
\begin{displaymath}
(Y\otimes X)*A \iso (X\otimes Y)*A \iso X*(Y*A)
\end{displaymath}
which lies in $\mathcal{M}$ as $X*\mathcal{K} \cie \mathcal{M}$. Thus $Y\otimes X$ lies in $\mathcal{T}_\mathcal{M}$.
\end{proof}

\begin{hyp}\label{hyp_tele}
We now, and for the rest of this section unless otherwise stated, ask more of $\mathcal{T}$ and $\mathcal{K}$: we suppose $\mathcal{T}$ has a model, so Theorem \ref{thm_general_ltg} applies, and that the assignments $\s$ and $\t$ of Definition \ref{defn_vis_sigmatau} provide a bijection between subsets of $\s\mathcal{K} \cie \Spc \mathcal{T}^c$ (which we give the subspace topology throughout) and localizing $\mathcal{T}$-submodules of $\mathcal{K}$. In particular, for any localizing submodule $\mathcal{M}$ of $\mathcal{K}$ there is an equality 
\begin{displaymath}
\mathcal{M} = \t(\s\mathcal{M}) = \{A\in \mathcal{K} \; \vert \; \supp A \cie \s\mathcal{M}\}.
\end{displaymath}
\end{hyp}

\begin{ex}
Let $R$ be a commutative noetherian ring with unit. Then the action of $D(R)$, the unbounded derived category of $R$, on itself satisfies the above hypotheses. More generally, these hypotheses still hold if one replaces $R$ by a noetherian scheme; this follows from Corollary 4.13 of \cite{AJS3} (see our Corollary \ref{cor_epicwin} for a proof using actions).
\end{ex}

\begin{lem}\label{lem_submod_im}
Suppose $\mathcal{M}$ is a localizing $\mathcal{T}$-submodule of $\mathcal{K}$. Then there is an equality of subcategories
\begin{displaymath}
\mathcal{M} = \mathcal{T}_\mathcal{M}*\mathcal{K}.
\end{displaymath}
\end{lem}
\begin{proof}
By Lemma \ref{general_im_tau} and $\t(\s\mathcal{M}) = \mathcal{M}$ we have
\begin{displaymath}
\mathcal{M} = \langle \mathit{\Gamma}_x\mathcal{K} \; \vert \; x\in \s\mathcal{M} \rangle_*.
\end{displaymath}
So by definition of $\mathcal{T}_\mathcal{M}$ the objects $\mathit{\Gamma}_x\mathbf{1}$ for $x\in \s\mathcal{M}$ lie in $\mathcal{T}_\mathcal{M}$. Thus $\mathcal{M} \cie \mathcal{T}_\mathcal{M}*\mathcal{K}$. That $\mathcal{T}_\mathcal{M}*\mathcal{K} \cie \mathcal{M}$ is immediate from the definition of $\mathcal{T}_\mathcal{M}$ giving the claimed equality.
\end{proof}

\begin{prop}\label{prop_smashing_orthog}
Suppose $\mathcal{T}$ satisfies the telescope conjecture and let $\mathcal{S} \cie \mathcal{K}$ be a smashing $\mathcal{T}$-submodule. If the inclusion  $\mathcal{T}_\mathcal{S} \to \mathcal{T}$ admits a right adjoint and
\begin{displaymath}
(\mathcal{T}_\mathcal{S})^\perp = \mathcal{T}_{\mathcal{S}^\perp}
\end{displaymath}
then $\mathcal{S}$ is generated by compact objects of $\mathcal{K}$.
\end{prop}
\begin{proof}
The subcategory $\mathcal{S}$ is, by assumption, a localizing submodule and as it is smashing $\mathcal{S}^\perp$ is also a localizing submodule by Lemma \ref{lem_smashing_orthogideal}. Thus Lemma \ref{lem_telecat_loc1} yields that both $\mathcal{T}_\mathcal{S}$ and $\mathcal{T}_{\mathcal{S}^\perp}$ are localizing $\otimes$-ideals of $\mathcal{T}$. By hypothesis the $\otimes$-ideals $\mathcal{T}_\mathcal{S}$ and $(\mathcal{T}_{\mathcal{S}})^\perp = \mathcal{T}_{\mathcal{S}^\perp}$ fit into a localization sequence. Hence $\mathcal{T}_\mathcal{S}$ is a smashing subcategory of $\mathcal{T}$ (this is well known, see for example \cite{BaRickard} Theorem 2.13). As the telescope conjecture is assumed to hold for $\mathcal{T}$ the subcategory $\mathcal{T}_\mathcal{S}$ is generated by objects of $\mathcal{T}^c$. By Lemma \ref{lem_submod_im} there is an equality of submodules
\begin{displaymath}
\mathcal{S} = \mathcal{T}_\mathcal{S} * \mathcal{K}
\end{displaymath}
which implies that $\mathcal{S}$ is generated by compact objects of $\mathcal{K}$: by Proposition \ref{prop_action_generation}, since $\mathcal{T}$ is rigidly-compactly generated and $\mathcal{T}_\mathcal{S}$ is generated by objects of $\mathcal{T}^c$, the subcategory $\mathcal{T}_\mathcal{S}*\mathcal{K}$ is generated by objects of $\mathcal{K}^c$.
\end{proof}

\begin{lem}\label{lem_subset_subcat}
Let $\mathcal{M}$ be a localizing submodule of $\mathcal{K}$ and let $W$ be a subset of $\Spc\mathcal{T}^c$ such that $W\intersec \s\mathcal{K} = \s\mathcal{M}$. Then there is a containment of $\otimes$-ideals of $\mathcal{T}$
\begin{displaymath}
\mathcal{T}_\mathcal{M} \supseteq \mathcal{T}_{W} = \{X\in \mathcal{T} \; \vert \; \supp X \cie W\}
\end{displaymath}
and 
\begin{displaymath}
\mathcal{T}_{W}*\mathcal{K} = \mathcal{M}.
\end{displaymath}
\end{lem}
\begin{proof}
It follows from the good properties of the support that $\mathcal{T}_{W}$ is a localizing $\otimes$-ideal of $\mathcal{T}$. Let $X$ be an object of  $\mathcal{T}_{W}$, let $A$ be an object of $\mathcal{K}$ and let $x$ be a point in $\Spec \mathcal{T}^c$. We have isomorphisms
\begin{displaymath}
\mathit{\Gamma}_x\mathbf{1}*(X*A) \iso (\mathit{\Gamma}_x\mathbf{1} \otimes X)*A \iso X*(\mathit{\Gamma}_x\mathbf{1}*A).
\end{displaymath}
The object $\mathit{\Gamma}_x\mathbf{1}\otimes X$ is zero if $x$ is not in $W$ and $\mathit{\Gamma}_x\mathbf{1}*A \iso 0$ if $x\notin \s\mathcal{K}$ so we see $\supp X*A$ is contained in $\s\mathcal{M}$. Thus $X*A$ is an object of $\mathcal{M} = \t\s\mathcal{M}$. It follows that $X$ is in $\mathcal{T}_{\mathcal{M}}$ and hence $\mathcal{T}_{W} \cie \mathcal{T}_{\mathcal{M}}$.

As $\supp \mathit{\Gamma}_x\mathbf{1} = \{x\}$ for $x\in \Spc \mathcal{T}^c$ by Corollary \ref{cor_supp_point1} we have $\mathit{\Gamma}_x\mathbf{1} \in \mathcal{T}_{W}$ for $x\in \s\mathcal{M}$. By the local-to-global principle (Theorem \ref{thm_general_ltg}) and $\t(\s\mathcal{M}) = \mathcal{M}$ we have
\begin{displaymath}
\mathcal{M} = \langle \mathit{\Gamma}_x\mathcal{K} \; \vert \; x\in \s\mathcal{M} \rangle_*
\end{displaymath}
so $\mathcal{T}_{W}*\mathcal{K} \supseteq \mathcal{M}$. We proved above that $\mathcal{T}_{W}\cie \mathcal{T}_\mathcal{M}$  which gives $\mathcal{T}_{W}*\mathcal{K} \cie \mathcal{M}$. Thus $\mathcal{T}_{W}*\mathcal{K} = \mathcal{M}$.
\end{proof}

\begin{lem}\label{lem_quot_supp_closed}
Suppose the support of any compact object of $\mathcal{K}$ is a specialization closed subset of $\s\mathcal{K}$. Then for any specialization closed subset $\mathcal{V}$ of $\Spc \mathcal{T}^c$, with complement $\mathcal{U}$, the support of every compact object of $L_{\mathcal{V}}\mathcal{K}$ is specialization closed in the complement $\mathcal{U} \intersec \s\mathcal{K}$ of $\mathcal{V}\intersec \s\mathcal{K}$ in $\s\mathcal{K}$ (with the subspace topology).
\end{lem}
\begin{proof}
Let us denote by $\pi$ the quotient functor $\mathcal{K} \to L_{\mathcal{V}}\mathcal{K}$. We assert it sends compact objects to compact objects. To see this is the case recall $\mathit{\Gamma}_{\mathcal{V}}\mathcal{K}$ has a generating set consisting of objects in $\mathcal{K}^c$ by Corollary \ref{cor_transfer_cg} so $\pi$ has a coproduct preserving right adjoint. The functor $\pi$ thus takes compact objects to compact objects by Theorem 5.1 of \cite{NeeGrot}.

Given any compact object $l$ of $L_{\mathcal{V}}\mathcal{K}$ there exists an object $k$ in $\mathcal{K}^c$ such that $l\oplus \S l$ is isomorphic to  $\pi k$ by \cite{NeeCat} Corollary 4.5.14. Thus
\begin{displaymath}
\supp l = \supp (l\oplus \S l) = \supp \pi k = \supp L_{\mathcal{V}} k = (\supp k) \intersec \mathcal{U}
\end{displaymath}
where this last equality is (4) of Proposition \ref{prop_abs_supp_prop}.  Thus $\supp l$ is specialization closed in $\mathcal{U}\intersec \s\mathcal{K}$ as $\supp k$ is specialization closed in $\s\mathcal{K}$.
\end{proof}

The next lemma is the key to our theorem on the relative telescope conjecture for good actions. Before stating and proving it we recall from \cite{BaSpec} Proposition 2.9 that the space $\Spc \mathcal{T}^c$ is $T_0$; given points $x,y \in \Spc \mathcal{T}^c$ we have $x = y$ if and only if $\mathcal{V}(x) = \mathcal{V}(y)$. In fact $\Spc\mathcal{T}^c$ is spectral in the sense of Hochster \cite{HochsterSpectral} so every irreducible closed subset has a unique generic point.

\begin{lem}\label{lem_imm_spec_closed}
Suppose the support of any compact object of $\mathcal{K}$ is a specialization closed subset of $\s\mathcal{K}$ and that for each irreducible closed subset $\mathcal{V} \cie \Spc\mathcal{T}^c$ there exists a compact object of $\mathcal{K}$ whose support is precisely $\mathcal{V}\intersec \s\mathcal{K}$. If $x$ and $y$ are distinct points of $\s \mathcal{K}$ with $y\in \mathcal{V}(x)$ then 
\begin{displaymath}
\langle \mathit{\Gamma}_{y'}\mathcal{K} \; \vert \; y' \in (\mathcal{V}(x)\intersec \mathcal{U}(y))\setminus \{x\}\rangle_\mathrm{loc} \notcie (\mathit{\Gamma}_x\mathcal{K})^\perp
\end{displaymath}
where $\mathcal{U}(y) = \{y' \in \Spc \mathcal{T}^c \; \vert \; y\in \mathcal{V}(y')\}$ is the complement of $\mathcal{Z}(y)$.
\end{lem}
\begin{proof}
By hypothesis there is a compact object $k$ of $\mathcal{K}$ satisfying
\begin{displaymath}
\supp k = \mathcal{V}(x)\intersec \s\mathcal{K}.
\end{displaymath}
The object $L_{\mathcal{Z}(y)}k$ is compact in $L_{\mathcal{Z}(y)}\mathcal{K}$ and has support
\begin{displaymath}
\supp L_{\mathcal{Z}(y)}k = (\supp k) \intersec (\Spc \mathcal{T}^c \setminus \mathcal{Z}(y))\intersec \s\mathcal{K} = \mathcal{V}(x) \intersec \mathcal{U}(y)\intersec \s\mathcal{K}
\end{displaymath}
by Proposition \ref{prop_abs_supp_prop}.

Suppose for a contradiction that
\begin{displaymath}
\langle \mathit{\Gamma}_{y'}\mathcal{K} \; \vert \; y' \in (\mathcal{V}(x)\intersec \mathcal{U}(y))\setminus \{x\}\rangle_\mathrm{loc} \cie (\mathit{\Gamma}_x\mathcal{K})^\perp.
\end{displaymath}
Consider the localization triangle for $L_{\mathcal{Z}(y)}k$
\begin{displaymath}
\mathit{\Gamma}_{\mathcal{Z}(x)} L_{\mathcal{Z}(y)}k \to L_{\mathcal{Z}(y)}k \to L_{\mathcal{Z}(x)}L_{\mathcal{Z}(y)}k \to \S\mathit{\Gamma}_{\mathcal{Z}(x)}L_{\mathcal{Z}(y)}k.
\end{displaymath}
We have, via Proposition \ref{prop_abs_supp_prop},
\begin{displaymath}
\supp L_{\mathcal{Z}(x)}L_{\mathcal{Z}(y)}k = \mathcal{U}(x) \intersec \mathcal{V}(x) \intersec \mathcal{U}(y) \intersec \s\mathcal{K} = \{x\}
\end{displaymath}
and
\begin{displaymath}
\supp \S\mathit{\Gamma}_{\mathcal{Z}(x)}L_{\mathcal{Z}(y)}k = \mathcal{Z}(x) \intersec \mathcal{V}(x) \intersec \mathcal{U}(y) \intersec \s\mathcal{K} = (\mathcal{V}(x) \intersec \mathcal{U}(y)\intersec \s\mathcal{K})\setminus \{x\}.
\end{displaymath}
So, as the local-to-global principle holds, the morphism $L_{\mathcal{Z}(x)}L_{\mathcal{Z}(y)}k \to \S\mathit{\Gamma}_{\mathcal{Z}(x)}L_{\mathcal{Z}(y)}k$ must be zero by our orthogonality assumption. This forces the triangle to split giving
\begin{displaymath}
L_{\mathcal{Z}(y)}k \iso L_{\mathcal{Z}(x)}L_{\mathcal{Z}(y)}k \oplus \mathit{\Gamma}_{\mathcal{Z}(x)}L_{\mathcal{Z}(y)}k.
\end{displaymath}
As $L_{\mathcal{Z}(y)}k$ is compact in $L_{\mathcal{Z}(y)}\mathcal{K}$ it follows that $L_{\mathcal{Z}(x)}L_{\mathcal{Z}(y)}k$ must also be compact. But we have already seen that the support of $L_{\mathcal{Z}(x)}L_{\mathcal{Z}(y)}k$ is $\{x\}$ which is not specialization closed in $\mathcal{U}(y)\intersec \s\mathcal{K}$. This yields a contradiction as by Lemma \ref{lem_quot_supp_closed} the compact objects in $L_{\mathcal{Z}(y)}\mathcal{K}$ have specialization closed support in $\mathcal{U}(y)\intersec \s\mathcal{K}$.
\end{proof}

\begin{lem}
Let $\mathcal{S}$ be a smashing $\mathcal{T}$-submodule of $\mathcal{K}$. Then 
\begin{displaymath}
\s \mathcal{S} \un \s \mathcal{S}^\perp = \s \mathcal{K} \quad \text{and} \quad \s \mathcal{S} \intersec \s \mathcal{S}^\perp = \varnothing.
\end{displaymath}
\end{lem}
\begin{proof}
Suppose $x$ is a point of $\s\mathcal{K}$ satisfying $x \in \s \mathcal{S} \intersec \s \mathcal{S}^\perp$. Then as we have assumed $\s$ and $\t$ are inverse bijections and $\mathcal{S}^\perp$ is a localizing submodule by Lemma \ref{lem_smashing_orthogideal} we would have
\begin{displaymath}
\mathit{\Gamma}_x \mathcal{K} \cie \mathcal{S} \intersec \mathcal{S}^\perp = 0.
\end{displaymath}
This contradicts $x\in \s \mathcal{K}$ as $x$ is a point of $\s \mathcal{K}$ if and only if $\mathit{\Gamma}_x \mathcal{K} \neq 0$.

We now show that every point of $\s\mathcal{K}$ lies in either $\s \mathcal{S}$ or $\s \mathcal{S}^\perp$. Let $x$ be a point of $\s \mathcal{K}$ and suppose $x\notin \s \mathcal{S}^\perp$. In particular $\mathit{\Gamma}_x\mathcal{K} \notcie \mathcal{S}^\perp$ so there is an object $X$ of $\mathit{\Gamma}_x\mathcal{K}$ with $\mathit{\Gamma}_\mathcal{S}X\neq 0$ where $\mathit{\Gamma}_\mathcal{S}$ is the acyclization functor associated to $\mathcal{S}$. Consider the localization triangle for $X$ associated to $\mathcal{S}$
\begin{displaymath}
\mathit{\Gamma}_\mathcal{S}X \to X \to L_\mathcal{S} X \to \S\mathit{\Gamma}_\mathcal{S} X.
\end{displaymath}
Applying $\mathit{\Gamma}_x$ we get another triangle
\begin{displaymath}
\mathit{\Gamma}_x \mathit{\Gamma}_\mathcal{S}X \to \mathit{\Gamma}_x X \to \mathit{\Gamma}_x L_\mathcal{S} X \to \S \mathit{\Gamma}_x\mathit{\Gamma}_\mathcal{S} X.
\end{displaymath}
Since $x\notin \s\mathcal{S}^\perp$ we have $\mathit{\Gamma}_x L_\mathcal{S}X \iso 0$. Hence
\begin{displaymath}
0\neq X \iso \mathit{\Gamma}_x X \iso \mathit{\Gamma}_x \mathit{\Gamma}_\mathcal{S} X
\end{displaymath}
so $\mathit{\Gamma}_x\mathcal{S}$ is not the zero subcategory and $x\in \s\mathcal{S}$. 

\end{proof}

\begin{lem}\label{lem_smashing_supp_comm}
Suppose the support of any compact object of $\mathcal{K}$ is a specialization closed subset of $\s\mathcal{K}$ and that for each irreducible closed subset $\mathcal{V}$ in $\Spc \mathcal{T}^c$ there exists a compact object of $\mathcal{K}$ whose support is precisely $\mathcal{V}\intersec \s\mathcal{K}$. Let $\mathcal{S} \cie \mathcal{K}$ be a smashing $\mathcal{T}$-submodule. Then the subset $\s\mathcal{S}$ is specialization closed in $\s\mathcal{K}$.
\end{lem}
\begin{proof}
We prove the lemma by contradiction. Let $x$ be a point of $\s\mathcal{S}$ and suppose $y$ is a point of $\mathcal{V}(x)\intersec \s\mathcal{K}$ which does not lie in $\s\mathcal{S}$. Then by the last lemma we must have $y\in \s\mathcal{S}^\perp$. We have assumed $\Spc \mathcal{T}^c$ is noetherian so there exists a point $x'$ of $\s\mathcal{S} \intersec \mathcal{U}(y)$ which is maximal with respect to specialization. We thus have 
\begin{displaymath}
((\mathcal{V}(x')\intersec \mathcal{U}(y))\setminus \{x'\})\intersec \s\mathcal{S} = \varnothing
\end{displaymath}
by virtue of the maximality of $x'$. From the previous lemma we deduce that every point of  $((\mathcal{V}(x')\intersec \mathcal{U}(y))\setminus \{x'\})$ lies in $\s\mathcal{S}^\perp$. As $\s$ and $\t$ are inverse there are containments 
\begin{displaymath}
\mathit{\Gamma}_{x'}\mathcal{K} \cie \mathcal{S} \quad \text{and} \quad \langle \mathit{\Gamma}_{y'}\mathcal{K} \; \vert \; y' \in (\mathcal{V}(x')\intersec \mathcal{U}(y))\setminus \{x'\}\rangle_* \cie \mathcal{S}^\perp
\end{displaymath}
the first as $x\in \s\mathcal{S}$ and the second by what we have just shown. Taking orthogonals in the first containment and combining we deduce that
\begin{displaymath}
\langle \mathit{\Gamma}_{y'}\mathcal{K} \; \vert \; y' \in (\mathcal{V}(x')\intersec \mathcal{U}(y))\setminus \{x'\}\rangle_* \cie \mathcal{S}^\perp \cie \mathit{\Gamma}_{x'}\mathcal{K}^\perp
\end{displaymath}
contradicting Lemma \ref{lem_imm_spec_closed} and completing the proof.


\end{proof}

\begin{thm}\label{thm_rel_tele}
Suppose the hypotheses of \ref{hyp_tele} hold, the support of any compact object of $\mathcal{K}$ is a specialization closed subset of $\s\mathcal{K}$ and that for each irreducible closed subset $\mathcal{V}$ of $\Spc\mathcal{T}^c$ there exists a compact object whose support is precisely $\mathcal{V}\intersec \s\mathcal{K}$. Then the relative telescope conjecture holds for $\mathcal{K}$ i.e., every smashing $\mathcal{T}$-submodule of $\mathcal{K}$ is generated, as a localizing subcategory, by compact objects of $\mathcal{K}$.
\end{thm}
\begin{proof}
Let $\mathcal{S}$ be a smashing submodule of $\mathcal{K}$. Recall from Lemma \ref{lem_subset_subcat} that there is an equality
\begin{equation}\label{eq_wtf}
\mathcal{T}_{W}*\mathcal{K} = \mathcal{S}
\end{equation}
for any $W\cie \Spc\mathcal{T}^c$ whose intersection with $\s\mathcal{K}$ is $\s\mathcal{S}$. By the lemma we have just proved the subset $\s\mathcal{S}$ is specialization closed in $\s\mathcal{K}$ so we can find a specialization closed subset $W$ of $\Spc\mathcal{T}^c$ with $W\intersec \s\mathcal{K} = \s\mathcal{S}$. As $W$ is specialization closed in $\Spc\mathcal{T}^c$ the tensor ideal $\mathcal{T}_{W}$ is generated by objects of $\mathcal{T}^c$. It then follows from the equality (\ref{eq_wtf}) that $\mathcal{S}$ is generated by objects of $\mathcal{K}^c$ - this last statement is the content of Proposition \ref{prop_action_generation}. 
\end{proof} 

\section{Working locally}\label{sec_locally}

We now show that the support theory we have developed is compatible with passing to quasi-compact open subsets of the spectrum; in particular, certain properties can be checked locally on an open cover.

Let $\mathcal{T}$ be a rigidly-compactly generated tensor triangulated category such that $\Spc\mathcal{T}^c$ is noetherian. We recall that, as $\Spc \mathcal{T}^c$ is noetherian, every open subset is quasi-compact. Let $U$ be an open subset with closed complement $Z$. There is an associated smashing localization sequence
\begin{displaymath}
\xymatrix{
\mathit{\Gamma}_Z\mathcal{T} = \mathcal{T}_Z \ar[r]<0.5ex>^(0.7){i_*} \ar@{<-}[r]<-0.5ex>_(0.7){i^!} & \mathcal{T} \ar[r]<0.5ex>^(0.3){p^*} \ar@{<-}[r]<-0.5ex>_(0.3){p_*} & L_Z\mathcal{T} = \mathcal{T}(U)
}
\end{displaymath}
where we have introduced the notation $\mathcal{T}(U)$ for the category on the right; we feel that this is worthwhile as when working locally it is better to keep open subsets in mind rather than their closed complements. Both $\mathcal{T}_Z$ and $\mathcal{T}(U)$ are tensor ideals and we recall that by definition
\begin{displaymath}
i_*i^! = \mathit{\Gamma}_Z\mathbf{1} \otimes(-) \quad \text{and} \quad p_*p^* = L_Z\mathbf{1}\otimes (-).
\end{displaymath}
By Thomason's localization theorem (see for example \cite{NeeGrot} Theorem 2.1) the subcategory of compact objects of $\mathcal{T}(U)$ is the idempotent completion of $\mathcal{T}^c/\mathcal{T}^c_Z$ i.e., it is precisely the subcategory $\mathcal{T}^c(U)$ of Balmer. By \cite{BaFilt} Proposition 2.15 the category $\mathcal{T}^c(U)$ is a rigid tensor category and so $\mathcal{T}(U)$ is a rigidly-compactly generated tensor triangulated category. We also wish to remind the reader that $\Spc\mathcal{T}^c(U)$ is naturally isomorphic to $U$ by \cite{BaGlue} Proposition 1.11. The quotient functor $p^*$ is monoidal and we will denote by $\mathbf{1}_U$ the tensor unit $p^*\mathbf{1}$ of $\mathcal{T}(U)$. 

We will use the notation introduced above throughout this section and it will be understood that $U$ carries the subspace topology. The category $\mathcal{T}(U)$ acts on itself giving rise to a support theory; in order to avoid confusion we will include $\mathbf{1}_U$ in the notation for acyclization, localization, and support functors this gives rise to, $\mathcal{T}(U)$ in the notation for the associated subcategories, and write the support as $\supp_{\mathcal{T}(U)}$.

Let us now recall that $p^*$ behaves nicely with respect to tensor idempotents in $\mathcal{T}$.

\begin{lem}\label{lem_idemp_proj}
Let $\mathcal{V}\cie \Spc \mathcal{T}^c$ be specialization closed. Then
\begin{displaymath}
p^*\mathit{\Gamma}_\mathcal{V}\mathbf{1} \iso \mathit{\Gamma}_{\mathcal{V}\intersec U}\mathbf{1}_U \quad \text{and} \quad p^*L_\mathcal{V}\mathbf{1} \iso L_{\mathcal{V}\intersec U}\mathbf{1}_U.
\end{displaymath}
\end{lem}
\begin{proof}
This is just a different way of stating \cite{BaRickard} Corollary 6.5.
\end{proof}

We next show the projection formula holds in this generality.

\begin{lem}\label{lem_projectionformula}
Suppose $X\in \mathcal{T}$ and $Y\in \mathcal{T}(U)$. Then there is an isomorphism
\begin{displaymath}
X\otimes p_*Y \iso p_*(p^*X\otimes Y).
\end{displaymath}
\end{lem}
\begin{proof}
As $Y$ is in $\mathcal{T}(U)$ we have $p^*p_*Y \iso Y$ and hence
\begin{displaymath}
p_*Y \iso p_*p^*p_* Y \iso L_Z\mathbf{1}\otimes p_* Y.
\end{displaymath} 
From this we see
\begin{align*}
\mathit{\Gamma}_Z\mathbf{1} \otimes X\otimes p_*Y &\iso X\otimes \mathit{\Gamma}_Z\mathbf{1}\otimes p_*Y \\
&\iso X\otimes \mathit{\Gamma}_Z\mathbf{1} \otimes L_Z\mathbf{1} \otimes p_*Y \\
&\iso 0
\end{align*}
showing $X\otimes p_*Y$ is in the image of $p_*$. Using this we deduce that
\begin{align*}
p_*(p^*X\otimes Y) &\iso p_*(p^*X\otimes p^*p_*Y) \\
&\iso p_*p^*(X\otimes p_*Y) \\
&\iso L_Z\mathbf{1}\otimes X\otimes p_*Y \\
&\iso X\otimes p_*Y
\end{align*}
which is the claimed isomorphism.
\end{proof}

It follows easily from these facts that one can work locally when considering the subcategories $\mathit{\Gamma}_x\mathcal{T}$ for $x\in \Spc \mathcal{T}^c$. 

\begin{prop}\label{prop_point_locala}
For all $x\in U$ there is an isomorphism
\begin{displaymath}
p_*\mathit{\Gamma}_x\mathbf{1}_U \iso \mathit{\Gamma}_x \mathbf{1}.
\end{displaymath}
\end{prop}
\begin{proof}
To see this is the case just note there are isomorphisms
\begin{align*}
p_*\mathit{\Gamma}_x\mathbf{1}_U &\iso p_*(\mathit{\Gamma}_{\mathcal{V}(x)\intersec U}\mathbf{1}_U \otimes L_{\mathcal{Z}(x)\intersec U}\mathbf{1}_U) \\
&\iso p_*(p^*\mathit{\Gamma}_{\mathcal{V}(x)}\mathbf{1} \otimes p^*L_{\mathcal{Z}(x)}\mathbf{1}) \\
&\iso p_*p^*(\mathit{\Gamma}_{\mathcal{V}(x)}\mathbf{1} \otimes L_{\mathcal{Z}(x)}\mathbf{1}) \\
&\iso L_Z\mathit{\Gamma}_x\mathbf{1} \\
&\iso \mathit{\Gamma}_x\mathbf{1}
\end{align*}
where we have used Lemma \ref{lem_idemp_proj} for the second isomorphism and the fact that $\mathit{\Gamma}_x\mathbf{1}\in L_Z\mathcal{T} = \mathcal{T}(U)$ for the final isomorphism.
\end{proof}

\begin{prop}\label{prop_point_localb}
For all $x\in U$ the functor $p_*$ induces an equivalence
\begin{displaymath}
\xymatrix{
\mathit{\Gamma}_x\mathcal{T} \ar[r]<0.75ex>^{p^*} \ar@{<-}[r]<-0.75ex>_{p_*} & \mathit{\Gamma}_x\mathcal{T}(U)
}.
\end{displaymath}
\end{prop}
\begin{proof}
The essential image of $p^*$ restricted to $\mathit{\Gamma}_x\mathcal{T}$ is $\mathit{\Gamma}_x\mathcal{T}(U)$ as we have isomorphisms
\begin{align*}
p^*(\mathit{\Gamma}_x\mathbf{1} \otimes X) &\iso p^*\mathit{\Gamma}_x\mathbf{1} \otimes p^*X \\
&\iso p^*p_*\mathit{\Gamma}_x\mathbf{1}_U \otimes p^*X \\
&\iso \mathit{\Gamma}_x\mathbf{1}_U \otimes p^*X
\end{align*}
where $X$ is any object of $\mathcal{T}$ and we have used the proposition we have just proved for the second isomorphism.

For $X$ in $\mathcal{T}$ we have, using the projection formula and Proposition \ref{prop_point_locala},
\begin{displaymath}
p_*(\mathit{\Gamma}_x\mathbf{1}_U \otimes p^*X) \iso p_*\mathit{\Gamma}_x\mathbf{1}_U \otimes X \iso  \mathit{\Gamma}_x\mathbf{1} \otimes X 
\end{displaymath}
showing the essential image of $p_*$ restricted to $\mathit{\Gamma}_x\mathcal{T}(U)$ is $\mathit{\Gamma}_x\mathcal{T}$.

Finally, as $p_*$ is fully faithful we have $p^*p_* \iso \id_{\mathcal{T}(U)}$ and $p_*p^* \iso \id_{\im p_*}$. From what we have just shown it is clear that this equivalence restricts to give the equivalence in the statement of the proposition.
\end{proof}

Let us now fix some action of $\mathcal{T}$ on a compactly generated triangulated category $\mathcal{K}$ and consider the relative version. For $U\cie \Spc \mathcal{T}^c$ as above we have a smashing localization sequence
\begin{displaymath}
\xymatrix{
\mathit{\Gamma}_Z\mathcal{K} \ar[r]<0.5ex>^(0.6){j_*} \ar@{<-}[r]<-0.5ex>_(0.6){j^!} & \mathcal{K} \ar[r]<0.5ex>^(0.3){q^*} \ar@{<-}[r]<-0.5ex>_(0.3){q_*} & L_Z\mathcal{K} = \mathcal{K}(U)
}
\end{displaymath}
by Lemma \ref{lem_action_transfer} and Corollary \ref{cor_transfer_cg}, where
\begin{displaymath}
j_*j^! = \mathit{\Gamma}_Z\mathbf{1}*(-) \quad \text{and} \quad q_*q^* = L_Z\mathbf{1}*(-).
\end{displaymath}
Our first observation is that $\mathcal{T}(U)$ acts on $\mathcal{K}(U)$ in a way which is compatible with the quotient functors.

\begin{prop}\label{prop_local_act}
There is an action $*_U$ of $\mathcal{T}(U)$ on $\mathcal{K}(U)$ defined by commutativity of the diagram
\begin{displaymath}
\xymatrix{
\mathcal{T}\times \mathcal{K} \ar[rr]^(0.4){p^*\times q^*} \ar[d]_{*} && \mathcal{T}(U) \times \mathcal{K}(U) \ar[d]^{*_U} \\
\mathcal{K} \ar[rr]_{q^*} && \mathcal{K}(U).
}
\end{displaymath}
\end{prop}
\begin{proof}
As in the diagram we define the action of $\mathcal{T}(U)$ on $\mathcal{K}(U)$ by setting, for $X\in \mathcal{T}$ and $A\in \mathcal{K}$,
\begin{displaymath}
p^*X *_U q^*A = q^*(X*A)
\end{displaymath}
and similarly for morphisms. This is well defined because, given $X'\in \mathcal{T}$, $A'\in \mathcal{K}$ with $p^*X\iso p^*X'$ and $q^*A \iso q^*A'$, then
\begin{align*}
q_*(p^*X *_U q^*A) &= q_*q^*(X*A) \\
&= L_Z(X*A)\\
&\iso L_ZX * L_ZA \\
&\iso L_ZX' * L_ZA' \\
&\iso q_*(p^*X' *_U q^*A')
\end{align*}
which implies $p^*X*_Uq^*A \iso p^*X'*_Uq^*A'$.

The associator and unitor are defined by the diagrams
\begin{displaymath}
\xymatrix{
(p^*X\otimes p^*Y)*_U q^*A \ar@{->}[d]_-\wr \ar[rr]^{a_U}_{\sim} && p^*X*_U(p^*Y*_U q^*A) \ar@{<-}[d]^-\wr \\
q^*((X\otimes Y)*A) \ar[rr]^\sim_{q^*a} && q^*(X*(Y*A))
}
\end{displaymath}
and
\begin{displaymath}
\xymatrix{
\mathbf{1}_U *_U q^*A \ar[r]^(0.6){l_U}_(0.6){\sim} \ar@{->}[d]_-\wr & q^*A \ar@{<-}[d]^-\wr \\
q^*(\mathbf{1}*A) \ar[r]^(0.6)\sim_(0.6){q^*l} & q^*A
}
\end{displaymath}
respectively for $X,Y\in \mathcal{T}$ and $A\in \mathcal{K}$. It is easily verified that $*_U$ fulfils the necessary conditions to be an action.
\end{proof}

We next prove the relative analogue of Proposition \ref{prop_point_localb}:

\begin{prop}\label{prop_point_local2}
For $x\in U$ there is an equivalence
\begin{displaymath}
\xymatrix{
\mathit{\Gamma}_x\mathcal{K} \ar[r]<0.75ex>^{q^*} \ar@{<-}[r]<-0.75ex>_{q_*} & \mathit{\Gamma}_x\mathcal{K}(U)
}.
\end{displaymath}
\end{prop}
\begin{proof}
The category $\mathit{\Gamma}_x\mathcal{K}$ is contained in $q_*\mathcal{K}(U)$ so $q^*$ is fully faithful when restricted to $\mathit{\Gamma}_x\mathcal{K}$. It just remains to note that for $A\in \mathcal{K}$
\begin{displaymath}
q^*(\mathit{\Gamma}_x\mathbf{1} * A) = p^*\mathit{\Gamma}_x\mathbf{1} *_U q^*A \iso \mathit{\Gamma}_x\mathbf{1}_U *_U q^*A
\end{displaymath}
so that $q^*\mathit{\Gamma}_x \mathcal{K} = \mathit{\Gamma}_x\mathcal{K}(U)$.
\end{proof}

\begin{rem}\label{rem_action_cover}
In particular, the last proposition implies that from an open cover $\Spc \mathcal{T}^c = \Un_{i=1}^n U_i$ we get an open cover
\begin{displaymath}
\s\mathcal{K} = \Un_{i=1}^n \s\mathcal{K}(U_i).
\end{displaymath}
\end{rem}

Now let us fix some cover $\Spc \mathcal{T}^c = \Un_{i=1}^n U_i$ by open subsets and denote the projections from $\mathcal{K}$ to $\mathcal{K}(U_i)$ by $q_i^*$. We will prove two results showing that one can deduce information about $\mathcal{K}$ from the corresponding statements for the $\mathcal{K}(U_i)$. First let us show that compact objects having (specialization) closed support is local in this sense.

\begin{lem}\label{lem_local_closed}
Suppose that for all $1\leq i \leq n$ and $a\in \mathcal{K}(U_i)$ compact the subset $\supp_{\mathcal{T}(U_i)}a$ is (specialization) closed in $U_i$. Then for all $b\in \mathcal{K}^c$ the subset $\supp b$ is (specialization) closed in $\Spc \mathcal{T}^c$.
\end{lem}
\begin{proof}
Let $b$ be compact in $\mathcal{K}$. Then
\begin{align*}
\supp b &= \Un_{i=1}^n (\supp b \intersec U_i) \\
&= \Un_{i=1}^n \{x\in U_i \; \vert \; \mathit{\Gamma}_x\mathbf{1}_U *_U q^*_i b\neq 0\} \\
&= \Un_{i=1}^n \supp_{\mathcal{T}(U_i)}q_i^* b
\end{align*}
as we have
\begin{displaymath}
\mathit{\Gamma}_x\mathbf{1}_U *_U q^*_ib = q^*_i(\mathit{\Gamma}_x b) \neq 0
\end{displaymath}
if and only if $x$ is in $\supp b \intersec U_i$. Now $q_i^*$ sends compacts to compacts as the associated localization is smashing, so by hypothesis each $\supp_{\mathcal{T}(U_i)}q^*_i b$ is (specialization) closed in $U_i$. Thus $\supp b$ is (specialization) closed in $\Spc \mathcal{T}^c$.
\end{proof}

\begin{rem}\label{rem_supp_local}
It is worth noting from the proof that for any $A\in \mathcal{K}$ there is an equality
\begin{displaymath}
\supp A = \Un_{i=1}^n \supp_{\mathcal{T}(U_i)} q_i^* A.
\end{displaymath}
\end{rem}

Finally we show it is also possible to check that $\s\mathcal{K}$ classifies localizing $\mathcal{T}$-submodules locally. It is easily seen that, provided $\mathcal{T}$ satisfies the local-to-global principle, a bijection between subsets of $\s\mathcal{K}$ and the collection of localizing submodules of $\mathcal{K}$ is equivalent to each of the $\mathit{\Gamma}_x\mathcal{K}$ being minimal in the following sense (cf.\ \cite{BIKstrat2} Section 4 and our Lemma \ref{general_im_tau}):

\begin{defn}\label{defn_min_subcat}
We say a localizing submodule $\mathcal{L}\cie \mathcal{K}$ is \emph{minimal} if it has no proper and non-trivial localizing submodules.
\end{defn}
By Proposition \ref{general_tau_inj} we have that $\s$ is left inverse to $\t$. To see $\t$ is an inverse to $\s$ one just needs to note that if the $\mathit{\Gamma}_x\mathcal{K}$ are minimal then the local-to-global principle completely determines any localizing submodule in terms of its support. In fact the converse is also true: such a bijection is easily seen to imply that the $\mathit{\Gamma}_x\mathcal{K}$ are minimal. Thus the following theorem should not come as a surprise.

\begin{thm}\label{thm_bijection_local}
Suppose $\mathcal{T}$ has a model and that there exists a cover $\Spc \mathcal{T}^c = \Un_i U_i$ for $i=1,\ldots,n$ such that the action of $\mathcal{T}(U_i)$ on $\mathcal{K}(U_i)$ yields bijections
\begin{displaymath}
\left\{ \begin{array}{c}
\text{subsets of}\; \s\mathcal{K}(U_i)
\end{array} \right\}
\xymatrix{ \ar[r]<1ex>^\t \ar@{<-}[r]<-1ex>_\s &} \left\{
\begin{array}{c}
\text{localizing submodules} \; \text{of} \; \mathcal{K}(U_i) \\
\end{array} \right\}.
\end{displaymath}
Then $\s$ and $\t$ give a bijection
\begin{displaymath}
\left\{ \begin{array}{c}
\text{subsets of}\; \s\mathcal{K}
\end{array} \right\}
\xymatrix{ \ar[r]<1ex>^\t \ar@{<-}[r]<-1ex>_\s &} \left\{
\begin{array}{c}
\text{localizing submodules} \; \text{of} \; \mathcal{K} \\
\end{array} \right\}.
\end{displaymath}
\end{thm}
\begin{proof}
By the discussion before the theorem it is sufficient to check that $\mathit{\Gamma}_x\mathcal{K}$ is minimal for each $x\in \s\mathcal{K}$. But for any such $x$ there exists an $i$ such that $x\in U_i$ and by Proposition \ref{prop_point_local2} the subcategory $\mathit{\Gamma}_x\mathcal{K}$ is equivalent to $\mathit{\Gamma}_x\mathcal{K}(U_i)$. This latter category is a minimal $\mathcal{T}(U_i)$-submodule by hypothesis and by the diagram of Proposition \ref{prop_local_act} this implies it is also minimal with respect to the action of $\mathcal{T}$.
\end{proof}

This machinery gives an easy proof of Corollary 4.13 of \cite{AJS3}. For a noetherian scheme $X$ let us denote by $D(X)$ derived category of $\str_X$-modules with quasi-coherent cohomology $D_{\QCoh}(\str_X$-$\Module)$. If $X$ is also separated this is equivalent to $D(\QCoh X)$ the derived category of quasi-coherent sheaves.

\begin{lem}\label{lem_open_loc_a}
Let $X$ be a noetherian scheme, let $U\cie X$ be an open set with complement $Z = X\setminus U$, and let $f\colon U\to X$ be the inclusion. If $E$ is an object of $D(X)$ then the map \mbox{$E\to \mathbf{R}f_*f^*E$} agrees with the localization map $E\to L_ZE$. In particular, $D(X)(U)$ is precisely $D(U)$.
\end{lem}
\begin{proof}
By definition the smashing subcategory $D_Z(X)$ giving rise to $L_Z$ is the localizing subcategory generated by the compact objects whose support is contained in $Z$. The kernel of $f^*$ is the localizing subcategory generated by those compact objects whose homological support is contained in $Z$. As these two notions of support coincide for compact objects of $D(X)$ (see for example \cite{BaSpec} Corollary 5.6) the lemma follows immediately.
\end{proof}

\begin{cor}\label{cor_epicwin}
Let $X$ be a noetherian scheme. Then, letting $D(X)$ act on itself, the assignments $\s$ and $\t$ give a bijection between subsets of $X$ and localizing $\otimes$-ideals of $D(X)$.
\end{cor}
\begin{proof}
This follows from Thomason's result that $\Spc D^{\mathrm{perf}}(X) \iso X$ \cite{Thomclass} and Neeman's classification of the localizing subcategories of $D(R)$ \cite{NeeChro}. One simply uses the last lemma to apply Theorem \ref{thm_bijection_local}.
\end{proof}

\begin{rem}
It is not hard to see that this also gives a classification of $\otimes$-ideals generated by perfect complexes in terms of specialization closed subsets of $X$.
\end{rem}

\section{Relation to the machinery of Benson, Iyengar, and Krause}\label{sec_BIK}
We now give an indication of the sense in which our actions may be regarded as an enhancement of the actions introduced in \cite{BIK}. Let $\mathcal{T}$ be a tensor triangulated category acting on a compactly generated triangulated category $\mathcal{K}$. First let us note that it is always the case that an action in our sense gives rise to an action by a ring on the central ring of $\mathcal{K}$.

\begin{defn}\label{defn_central_ring}
Let $\mathcal{T}$ be a triangulated category. The \emph{graded centre} (or \emph{central ring}) of $\mathcal{T}$ is the graded abelian group
\begin{displaymath}
Z^*(\mathcal{T}) = \bigoplus_{n} Z^n(\mathcal{T}) = \bigoplus_{n} \{\a\colon \id_\mathcal{T} \to \S^n \; \vert \; \a\S = (-1)^n\S\a\}
\end{displaymath}
where $n$ ranges over the integers, which is given the structure of a graded commutative ring by composition of natural transformations.
\end{defn}

\begin{rem}
Using the words ring and group above is somewhat abusive as the centre of $\mathcal{T}$ may not form a set (we do not assume $\mathcal{T}$ essentially small). However, this is not a problem if one only wishes to consider the images of genuine rings.
\end{rem}

\begin{lem}\label{lem_bik_map}
An action $\mathcal{T}\times\mathcal{K} \stackrel{*}{\to} \mathcal{K}$ induces a morphism of rings
\begin{displaymath}
\End^*_\mathcal{T}(\mathbf{1}) \to Z^*(\mathcal{K}).
\end{displaymath}
\end{lem}
\begin{proof}
Given $f \in \Hom(\mathbf{1}, \S^i \mathbf{1})$ we send it to the natural transformation whose component at $A\in \mathcal{K}$ is 
\begin{displaymath}
\xymatrix{
A \ar[r]^(0.4){\sim} & \mathbf{1}*A \ar[r]^{f*1_A} & \S^i \mathbf{1}*A \ar[r]^(0.6){\sim} & \S^i A.
}
\end{displaymath}
This is natural by our coherence conditions. It is a standard fact, given our compatibility conditions, that the graded endomorphism ring of the unit is graded commutative (see for example \cite{suarez-alvarez}) from which it is straightforward that this is a map of graded commutative rings.
\end{proof}

Thus provided $\End^*_\mathcal{T}(\mathbf{1})$ is noetherian one is in a position to apply the machinery of Benson, Iyengar, and Krause. In fact this is discussed in Section 8 of \cite{BIK} for the case of tensor triangulated categories acting on themselves and it is shown in Section 9 that for the derived category of a noetherian ring one recovers the classical notion of supports from their construction. In fact, for a noetherian ring $R$, using the lemma above to move from an action by $D(R)$ to an action by $R$ does not change the support theory.

\begin{prop}\label{rem_BIK_reln}
Let $R$ be a noetherian ring and suppose
\begin{displaymath}
D(R) \times \mathcal{K} \stackrel{*}{\to} \mathcal{K}
\end{displaymath}
is an action of $D(R)$ on a compactly generated triangulated category $\mathcal{K}$. Then the support theory of Section \ref{sec_supports} agrees with the support theory given in \cite{BIK} via the morphism of Lemma \ref{lem_bik_map}.
\end{prop}
\begin{proof}
By \cite{BIK} Theorem 6.4 the subcategories giving rise to supports in the sense of Benson, Iyengar, and Krause are generated by certain Koszul objects: if $\mathcal{V}\cie \Spec R$ is specialization closed then their subcategory $\mathcal{K}_\mathcal{V}$ is easily seen to be generated by the objects
\begin{displaymath}
\{ K(\mathfrak{p})* a \; \vert \; a\in \mathcal{K}^c, \; \mathfrak{p}\in \mathcal{V}\}.
\end{displaymath}
As $\{\S^iK(\mathfrak{p}) \; \vert \; \mathfrak{p}\in \mathcal{V}, i\in \int\}$ is a generating set for $\mathit{\Gamma}_\mathcal{V}D(R)$ we see, by Remark \ref{rem_action_generation} and the corollary following it, that the localizing subcategories $\mathcal{K}_\mathcal{V}$ and $\mathit{\Gamma}_\mathcal{V}\mathcal{K}$ agree. Thus our support functors are precisely those of Benson, Iyengar, and Krause in the case that the derived category of a noetherian ring acts.
\end{proof}


  
  \bibliography{greg_bib}

\end{document}